\documentclass{amsart}
\usepackage{amssymb}
\usepackage{amsmath}
\usepackage{amsfonts}
\usepackage{diagbox}
\usepackage{mathrsfs}
\usepackage{booktabs}
\usepackage{multirow}
\usepackage{color}
\usepackage{url}
\usepackage{ulem}
\usepackage{algorithm}
\usepackage{algorithmic}
\usepackage{caption}
\usepackage{subcaption}
\usepackage{graphicx}

\newcommand{\rA}{\mathscr{A}}

\newcommand{\op}{\operatorname}
\newcommand{\re}{\mathbb{R}}

\newcommand{\F}{\mathcal{F}}

\renewcommand{\P}{\mathcal{P}}

\newcommand{\six}{\sigma_{i}(\bar X)}
\newcommand{\sx}{\sigma(\bar X)}
\newcommand{\shx}{\sigma(\hat X)}

\newcommand{\s}{\sigma(X)}

\def\bt{\beta}

\newcommand{\reff}[1]{(\ref{#1})}

\newcommand{\bnum}{\begin{enumerate}}
\newcommand{\enum}{\end{enumerate}}

\newcommand{\bit}{\begin{itemize}}
\newcommand{\eit}{\end{itemize}}

\newcommand{\be}{\begin{equation}}
\newcommand{\ee}{\end{equation}}

\newcommand{\baray}{\begin{array}}
\newcommand{\earay}{\end{array}}

\newcommand{\bca}{\begin{cases}}
\newcommand{\eca}{\end{cases}}

\newcommand{\bcen}{\begin{center}}
\newcommand{\ecen}{\end{center}}

\newcommand{\bbm}{\begin{bmatrix}}
\newcommand{\ebm}{\end{bmatrix}}

\newcommand{\bpm}{\begin{pmatrix}}
\newcommand{\epm}{\end{pmatrix}}

\newcommand{\btab}{\begin{tabular}}
\newcommand{\etab}{\end{tabular}}

\newtheorem{theorem}{Theorem}[section]
\newtheorem{pro}[theorem]{Proposition}
\newtheorem{prop}[theorem]{Proposition}
\newtheorem{lem}[theorem]{Lemma}
\newtheorem{lemma}[theorem]{Lemma}

\newtheorem{defi}[theorem]{Definition}
\newtheorem{example}[theorem]{Example}

\newtheorem{remark}[theorem]{Remark}

\newcommand{\bm}[1]{{\mbox{\boldmath $#1$}}}
\newcommand*{\R}{\mathbb{R}}

\def\beq#1{\begin{equation}\label{#1}}
\def\eeq{\end{equation}}
\def\bep{\begin{proof}}

\def\ep{\end{proof}}
\def\bt{\begin{theorem}}
\def\et{\end{theorem}}
\def\bl{\begin{lemma}}
\def\el{\end{lemma}}
\def\reff#1{(\ref{#1})}

\def\ignore#1{}

\def\bea#1{\begin{array}{#1}}
\def\ea{\end{array}}

\input

\makeatletter
\@namedef{subjclassname@2020}{%
	\textup{2020} Mathematics Subject Classification}
\makeatother

\begin{document}
\setcounter{page}{1}

\title{A  smoothing proximal gradient algorithm for matrix rank minimization problem }

\author{Quan Yu, Xinzhen Zhang}
\address{School of Mathematics, Tianjin University, Tianjin 300350, China.}
\email{QuanYu527@163.com}

\address{School of Mathematics, Tianjin University, Tianjin 300350, China.
} \email{xzzhang@tju.edu.cn}

\begin{abstract}
	In this paper, we study the low-rank matrix minimization problem, where the loss
	function is convex but nonsmooth and the penalty term is defined by the cardinality function. We first introduce an exact continuous relaxation, that is, both problems have the same minimzers and the same optimal value. In particular, we introduce a class of  lifted stationary point of the relaxed problem and show that any local minimizer of the relaxed problem must be a lifted stationary point. In addition, we derive lower bound property for the nonzero singular values of the lifted stationary point and hence also of the local minimizers of the relaxed problem. Then the smoothing proximal gradient (SPG) algorithm is proposed to find a lifted stationary point of the continuous relaxation model. Moreover, it is shown that the whole sequence generated by SPG algorithm converges to a lifted stationary point. At last, numerical examples show the efficiency of the SPG algorithm.	
\end{abstract}

\keywords{low-rank approximation, nonsmooth convex loss function, smoothing method}

\subjclass[2020]{15A03,15A83,90C30, 65K05}

\maketitle

\section{Introduction}
Over the last decade, finding a low-rank matrix solution to a system or low-rank matrix optimization problem have received more and more attention. Numerous optimization models and methods have been proposed in \cite{FHB01,FRW11,HH09,MGC11,MXZZ13,WYWY12}. In this paper, we consider the following matrix rank minimization problem with cardinality penalty, that is,
\begin{equation}\label{l0}
\min \F_{l_0}(X):=f(X)+\lambda\cdot \op{rank}(X)=f(X)+\lambda\|\s\|_{0},
\end{equation}
where $ X \in \mathbb R^{m\times n}\,(n\le m) $ and
$\sigma(X):=\left(\sigma_{1}(X), \ldots, \sigma_{n}(X)\right)^T$ is a vector composed of $X$ 's singular values with $\sigma_{1}(X) \geq \ldots \geq \sigma_{n}(X)\geq 0$. Furthermore, $ f: \mathbb{R}^{m\times n} \rightarrow[0, \infty)$ is convex (not necessarily smooth) and
$\lambda$ is a positive parameter.

One application of problem \eqref{l0} is the low-rank matrix recovery problem\cite{LL16,LZL17,MLH17,SWKT19}. 
To solve such problem, traditional algorithms are always based on $l_2$(or Frobenius)-nuclear model, that is, the loss function is a $l_2$-norm for vector case or Frobenius norm for matrix case, and $\op{rank}(X)$ is relaxed as a matrix nuclear norm.
   However, these models are sensitive to non-Gaussian noise with outliers \cite{HWLQC20,ZLSYO12,ZMXZY15,ZS18}. To overcome this drawback, the $l_1$ model is considered in the problem with the outlier-resistant loss function. For example, 
the following loss function is considered in the low-rank matrix recovery problem
\begin{equation}\label{l1-loss}
	f(X)=\|\rA(X)-b\|_{1},
\end{equation}
where 
the linear map $\mathscr{A}: \mathbb{R}^{m \times n} \rightarrow \mathbb{R}^{p}$ and vector $b \in \mathbb{R}^{p}$ are given. Obviously, $ f $ is convex but not smooth.
Considering low-rank matrix completion problem,  a special case of low-rank matrix recovery problem, the corresponding loss function $f(X)$ can be written as
\begin{equation*}
	f(X)=\|P_\Omega\left(X-M \right) \|_1,
\end{equation*}
where $M \in \R^{m \times n}$ is a known matrix, $\Omega$ is an index set which locates the observed data, $ \P_\Omega $ is a linear operator that extracts the entries in $\Omega$ and fills the entries not in $\Omega$ with zeros. In the robust principal component analysis (RPCA) problem \cite{BZ14,CCYZZ,PKC16,XCS12}, the loss function $f(X)$
is adopted as
\begin{equation}\label{eq:rpca}
	f(X)=\|L-X\|_{1},
\end{equation}
where $L \in \mathbb{R}^{m \times n}$ denotes the observed data. The RPCA problem aims to decompose the matrix $L$ as the sum of a low-rank matrix $X$ and a sparse matrix $E=L-X \in \mathbb{R}^{m \times n}$.

It is known that  matrix rank function is nonconvex and  nonsmooth. In the matrix rank minimization problem, one of common used convex relaxations of rank function is matrix nuclear norm.
Although the methods based  nuclear norm relaxation
have strong theoretical guarantees,  the obtained approximation solutions under certain incoherence assumptions are usually hard to satisfy in real applications \cite{CR09,CT10}. In other words, the nuclear norm is not a perfect approximation to the rank function.

In \cite{BC20}, the capped $l_1$ function, a continuous relaxation of $l_0$ function, was adopted in penalized sparse regression problem with some advantages.
Furthermore, a smoothing proximal gradient (SPG) algorithm with global convergence was proposed there. More recently, such technique was applied to group sparse optimization for images recovery in \cite{PC21}. It is well-known that the matrix norm can be expressed as a vector norm of  the singular value vector.
Motivated by these, we consider whether such SPG algorithm can be generalized from sparse regression problem  to low-rank matrix minimization  or not.

For this aim, let $\Phi(X)=\sum_{i=1}^{n} \phi\left(\sigma_i\left(X\right)\right)$  be a continuous relaxation of the rank function with the capped-$\ell_{1}$ function $\phi$ given by
\begin{equation}\label{phi}
	\phi(t)=\min \{1, t/ \nu\},\quad t\ge0 ,
\end{equation} where $\nu>0$ is  a parameter.
Based on $\Phi(X)$, we consider the following continuous optimization problem for solving \eqref{l0}:
\begin{equation}\label{lc}
\min \mathcal{F}(X):=f(X)+\lambda \Phi(X).
\end{equation}

Our contributions are as follows.
We first present the continuous relaxation problem (\ref{lc}) of problem (\ref{l0}), which are shown to have the same global optimizers. Furthermore, the local minimizer of \reff{lc} is a lifted stationary point of \reff{lc} with an expected lower bound property of singular values. Then an SPG algorithm is proposed to get a lifted stationary point of \reff{lc}  with global convergence.

\bm{Notations.} We denote $[n]=\{1,2, \ldots,n\}$ and $\mathbb{D}^{n}=\left\{d \in \mathbb{R}^{n}: d_{i} \in\{1,2\}, i\in [n]\right\}.$
 The space of $m \times n$ matrices is denoted by $\mathbb R^{m \times n}$. For a given matrix $X \in \mathbb{R}^{m\times n}$, $\mathbb{B}_{\delta}(X)$ denotes the open ball centered at $X$ with radius $\delta$. In addition, $ \mathscr D(x) $ denotes a diagonal matrix generated by vector $ x $, whose dimension shall be clear from the context. Denote $\mathbb Q^{m}$ the set of $m\times m$-dimension unitary orthogonal matrix. Let $ E_i=\mathscr D(e_i)$, where $e_i$ is a unit vector whose $i$th entry is $1$.

For any given  $X,Y \in \mathbb R^{m \times n}$,
the standard inner product of $X$ and $Y$ is denoted by $\langle X, Y\rangle$, that is, $\langle X, Y\rangle=\operatorname{tr}\left(X Y^{T}\right)$, where $\operatorname{tr}(\cdot)$ denotes the trace of a matrix.
The Frobenius norm of $X$ is denoted by $\|X\|_{F},$ namely, $\|X\|_{F}=\sqrt{\operatorname{tr}\left(X X^{T}\right)}$.
Denote $\sigma(X)=\left(\sigma_{1}(X), \ldots, \sigma_{n}(X)\right)^{T}$ and
$$
\mathcal{M}(X)=\left\{(U, V) \in \mathbb Q^{m} \times \mathbb Q^{n}:\,\, X=U \mathscr D\left( \sigma(X)\right)  V^{T}\right\}.
$$

\section{An exact continuous relaxation for \eqref{l0}}
\setcounter{equation}{0}
In this section, we present the relationships between \eqref{l0} and \eqref{lc}. Without specific explanation, Assumptions 1 and 2 are assumed throughout the paper.
\par Assumption 1. $f$ is Lipschitz continuous with Lipschitz constant $L_{f}$.
\par Assumption 2. Positive parameter $\nu$ in \eqref{phi} satisfies $\nu<\bar{\nu}:=\lambda / L_{f}$.

\subsection{Lifted stationary points of \eqref{lc}}
Clearly,  $\phi$ in \eqref{phi} can be rewritten as a DC function, i.e.,
$$
\phi(t)=\frac{t}{\nu}-\max \left\{\theta_{1}(t), \theta_{2}(t)\right\}
$$
with $\theta_{1}(t)=0$ and $\theta_{2}(t)=t / \nu-1$. Denote
\begin{equation}\label{dc}
\mathcal{D}(t)=\left\{i \in\{1,2\}: \theta_{i}(t)=\max \left\{\theta_{1}(t), \theta_{2}(t)\right\}\right\}.	
\end{equation}
Following Theorem 3.7 in \cite{LS05}, the Clarke subdifferential of $\Phi$ at $X$ is given by
$$
\partial \Phi(X)=\left\{U \mathscr{D}(x) V^{T}: x \in \partial \sum_{i=1}^{n} \phi\left(\sigma_i\left(X\right)\right),\,(U, V) \in {\mathcal{M}}(X)\right\},
$$
where $\partial \phi(x)$ is the Clarke subdifferential \cite{C90} of $\phi(x)$. Then we have the following definition.

\begin{defi}\label{def:ls}
We say that $X $ is a lifted stationary point of \eqref{lc} if there exist $d_{i} \in \mathcal{D}\left(\sigma_{i}(X)\right)$ for all $i\in [n]$ such that
\begin{equation}\label{ls}
\lambda  \sum_{i=1}^{n} \theta_{d_{i}}^{\prime}\left(\sigma_{i}(X)\right) E_{i} \in \left\lbrace U^T\partial f(X)V+\frac{\lambda}{\nu} \mathscr{D}\left(\partial \left\|{\sigma}(X)\right\|_1\right):(U, V) \in \mathcal{M}\left(X\right)\right\rbrace,	
\end{equation}
where $ \sigma_i(X) $ is the $i$th largest singular value of $ X $.

If \eqref{ls} holds for all $d_{i} \in \mathcal{D}\left(\sigma_{i}(X)\right),~ i\in[n]$, then we call $X$ a d-stationary point.
\end{defi} 

\subsection{Characterizations of lifted stationary points of \eqref{lc}}
We first show that for any element in $\mathcal{D}\left(\sigma_{i}(X)\right)$ satisfying \eqref{ls} is unique and well defined.
\begin{prop}\label{p_d}
If $\bar{X}$ is a lifted stationary point of \eqref{lc}, then the vector $d^{\bar{X}}= \left(d_{1}^{\bar{X}}, \ldots, d_{n}^{\bar{X}}\right)^{T} \in \prod_{i=1}^{n} \mathcal{D}\left(\sigma_{i}(\bar X)\right)$ satisfying \eqref{ls} is unique. In particular, for $i\in [n]$,
\begin{equation}\label{d}
d_{i}^{\bar{X}}=\left\{\begin{array}{ll}
	1 & \text { if }\sigma_{i}(\bar X)<\nu, \\
	2 & \text { if } \sigma_{i}(\bar X) \geq \nu.
\end{array}\right.	
\end{equation}
\end{prop}
\begin{proof}
For case of $\sigma_{i}(\bar X) \neq \nu,$ the statement in this proposition follows. Hence, it suffices to consider the index $i$ satisfying $\sigma_{i}(\bar X)=\nu$.

 Now we assume that $d_{i}^{\bar{X}}=1$ by contradiction when $\sigma_{i}(\bar X)=\nu$.  By \eqref{ls}, there exists $\xi(\bar{X}) \in \partial f(\bar{X})$ such that $0=\left( U^T\xi(\bar{X})V\right)_{ii}+\lambda / \nu$, where $(U, V) \in \mathcal{M}\left(\bar X\right)$. Then, $\lambda /\nu = \left| {{{\left( {{U^T}\xi (\bar X)V} \right)}_{ii}}} \right| \le {\left\| {{U^T}\xi (\bar X)V} \right\|_F} = {\left\| {\xi (\bar X)} \right\|_F} \le {L_f}$. This leads to a contradiction to $\nu<\lambda / L_{f}$.  Then, we can assert that $d_{i}^{\bar{X}}=2$, and hence \eqref{d} holds for $\sigma_{i}(\bar X)=\nu$.
\end{proof}

For a given $d=\left(d_{1}, \ldots, d_{n}\right)^{T} \in \mathbb{D}^{n},$ we define
\begin{equation}\label{eq:Phi_d}
\Phi^{d}(X):=\sum_{i=1}^{n}\sigma_{i}(X) / \nu-\sum_{i=1}^{n} \theta_{d_{i}}\left(\sigma_{i}(X)\right).	
\end{equation}
It is easy to see that
$$\Phi(X)=\min_{d \in \mathbb{D}^{n}}\Phi^d(X),~\forall X\in\mathbb{R}^{m\times n}.$$
Furthermore, for a fixed $\bar{X},~ \Phi(\bar X)=\Phi^{d^{\bar{X}}}(\bar X)$ with $d^{\bar X}$ defined in \eqref{d}.
We next show that any local minimizer of \eqref{lc} is a lifted stationary point of the problem.
\begin{theorem}\label{th:lsp}
	Suppose that $\bar X$ is a local minimizer of problem \eqref{lc}. Then $\bar X$ is a lifted stationary point of $\eqref{lc}$, that is, $\eqref{ls}$ holds at $\bar X$.
\end{theorem}
\begin{proof}
	Since $\bar X$ is a local minimizer of \eqref{lc} satisfying \reff{lc}, it gives
\begin{equation}\label{eq:pdx}
\begin{aligned}
f(\bar X) + \lambda {\Phi ^{{d^{\bar X}}}}\left( {\bar X} \right) &= f(\bar X) + \lambda \Phi \left( {\bar X} \right)  \\
&\le f(X) + \lambda \Phi \left( X \right) \le f(X) + \lambda {\Phi ^{{d^{\bar X}}}}\left( X \right),\quad \forall X \in \mathbb{B}_{\varrho}(\bar{X}),
\end{aligned}	
\end{equation}
where the first equality comes from $ \Phi^{d^{\bar{X}}}(\bar X) = \Phi(\bar X)$, $d^{\bar X}$ is defined as in \eqref{d} and the last inequality is due to $\Phi^{d^{\bar{X}}}(X) \geq \Phi(X),\,\forall X \in \mathbb{R}^{m\times n}$. Then $ \bar X $ is a local minimizer of the problem
	\begin{equation}
	\min _{X} f\left(X\right)+\lambda {\Phi ^{{d^{\bar X}}}}\left( X \right).
	\end{equation}
	Hence, there exists some $(\bar{U}, \bar{V}) \in$
	$\mathcal{M}\left(\bar X\right)$ such that
	\begin{equation}\label{}
	0 = \partial f\left( {\bar X} \right) + \lambda {\bar U}\left( {\frac{1}{\nu }{\mathscr{D}}\left( {\partial {{\left\| {\sigma (\bar X)} \right\|}_1}} \right) - \sum\limits_{i = 1}^n {\theta _{{d_i^{\bar X}}}^\prime } \left( {{\sigma _i}(\bar X)} \right){E_i}} \right)\bar V^T,
	\end{equation}
	which implies $\eqref{ls}$ at $\bar X$.
\end{proof}
Now we show a lower bound property of the lifted stationary points of \eqref{lc}, which is similar to Lemma 2.3 in \cite{BC20}. For the ease of the reader, we present the proof here.
\begin{lem}\label{lem:lb}
If $ \bar{X} $ is a lifted stationary point of \eqref{lc}, then it holds that
\begin{equation}\label{eq:lb}
\six \in \left[ 0, \nu\right)  \Rightarrow  \six=0, \quad  i\in[n].	
\end{equation}
\end{lem}
\begin{proof}
Suppose $\bar{X}$ is a lifted stationary point of \eqref{lc}. Assume that $\six \in(0, \nu)$ for some $i \in[n] .$ Then, $d_{i}^{\bar{X}}=1$. By Definition \ref{def:ls}, there exists $\xi(\bar{X}) \in \partial f(\bar{X})$ such that $0=\left( U^T\xi(\bar{X})V\right)_{ii}+\lambda / \nu$, where $(U, V) \in \mathcal{M}\left(\bar X\right)$. Then, $\lambda /\nu = \left| {{{\left( {{U^T}\xi (\bar X)V} \right)}_{ii}}} \right| \le {\left\| {{U^T}\xi (\bar X)V} \right\|_F} = {\left\| {\xi (\bar X)} \right\|_F} \le {L_f}$, which leads to a contradiction to $\nu<\lambda / L_{f}$. Thus, for any $i \in[n],\, \six \in\left[0, \nu\right) $ implies that $\six=0$.
\end{proof}

\subsection{Relationship between \eqref{l0} and \eqref{lc}}
This subsection  presents the relationship between problem \eqref{l0} and its continuous relaxation \eqref{lc}. According to  the lower bound property of the lifted stationary points of \eqref{lc} in Lemma \ref{lem:lb}, we are ready to link \eqref{l0} and \eqref{lc} by the  following two results.

\begin{theorem}\label{tho:gm}
$\bar{X} $ is a global minimizer of \eqref{l0} if and only if it is a global minimizer of \eqref{lc}. Moreover, problems \eqref{l0} and \eqref{lc} have the same optimal value.
\end{theorem}
\begin{proof}
Let $\bar{X}$ be a global minimizer of \eqref{lc}, then $\bar{X}$ is a lifted stationary point of \eqref{lc} from Theorem \ref{th:lsp}. By \eqref{eq:lb}, it follows $\Phi(\bar X)=\|\sx\|_{0}$. Then,
\begin{equation*}
\begin{aligned}
f(\bar{X})+\lambda\|\sx\|_{0}&=f(\bar{X})+\lambda \Phi(\bar X) \\
&\leq f(X)+\lambda \Phi(X)	\\
& \leq f(X)+\lambda\|\s\|_{0}, 	
\end{aligned}
\end{equation*}
where the last inequality comes from $\Phi(X) \leq\|\s\|_{0},~ \forall X \in \mathbb{R}^{m\times n}$. Thus, $\bar{X}$ is a global minimizer of \eqref{l0}.

Next, suppose $\bar{X}$ is a global minimizer of \eqref{l0} but not a global minimizer of \eqref{lc}. Assume that $\hat X$ is a global minimizer of \eqref{lc} satisfying
$$
f(\hat{X})+\lambda \Phi(\hat X)<f(\bar{X})+\lambda \Phi(\bar X).
$$ As shown earlier, $\Phi(\hat X)=\|\sigma(\hat X)\|_{0}$.
Together with  $\Phi(\bar X) \leq\|\sx\|_{0}$, we have $f(\hat{X})+\lambda\|\shx\|_{0}<f(\bar{X})+\lambda\|\sx\|_{0},$ which leads
to a contradiction. Thus, any global minimizer of \eqref{l0} must be a global minimizer of \eqref{lc}.

Lemma \ref{lem:lb} ensures that problems \eqref{l0} and \eqref{lc} have the same optimal value.
\end{proof}
%
%
\begin{prop}\label{pro:lm}
If $\bar{X}$ is a local minimizer of \eqref{lc}, then it is a local minimizer of \eqref{l0}, and the objective functions of (\ref{l0}) and (\ref{lc}) have the same value at $\bar{X},$ i.e., $\mathcal{F}_{\ell_{0}}(\bar{X})=$ $\mathcal{F}(\bar{X})$.
\end{prop}

The proof is similar to the first part of Theorem \ref{tho:gm}
and hence we omit it here.

To end this subsection, we present Figure \ref{fig:link} to summarize the relationship between problems \eqref{l0} and \eqref{lc}.
 \begin{figure}[htbp]
	\centering
	\begin{subfigure}[b]{1\linewidth}
		\begin{subfigure}[b]{\linewidth}
			\centering
			\includegraphics[width=\linewidth]{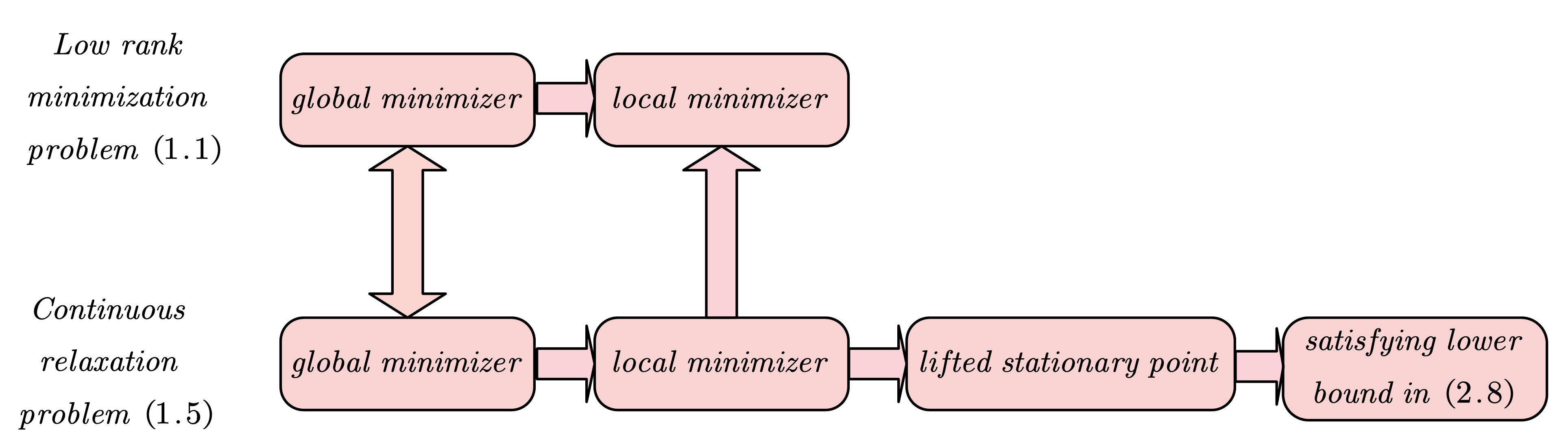}
		\end{subfigure}  	
	\end{subfigure}
	\vfill
	\caption{Links between problems \eqref{l0} and \eqref{lc}.}
	\label{fig:link}
\end{figure}
$
$

\section{Numerical algorithm and its convergence analysis}

In this section, we establish a numerical algorithm  to find a lifted stationary point of \eqref{lc}. We first introduce some useful preliminary results on smoothing methods and the proximal gradient algorithm, then we propose  a proximal gradient algorithm based on the smoothing method. After that we present  the convergence of the proposed algorithm.

\subsection{Smoothing approximation method and proximal gradient method}
 Smoothing approximation method is a common-used numerical method for solving nonsmooth optimization problems. For more details, see \cite{NW06} and references therein. For the sake of completeness, we recall a class of smoothing functions for $ f(X) $ in \eqref{lc}.

\begin{defi}\label{def:smoo}
We call $\tilde{f}: \mathbb{R}^{m\times n} \times[0, \bar{\mu}] \rightarrow \mathbb{R}$ with $\bar{\mu}>0$ a smoothing function of the convex function $f$ in \eqref{lc}, if $\tilde{f}(\cdot, \mu)$ is continuously differentiable in $\mathbb{R}^{m\times n}$ for any fixed $\mu>0$ and satisfies the following conditions:
\begin{itemize}
	\item [(i)] $\lim \limits_{X \rightarrow \bar X, \mu \downarrow 0} \tilde{f}(X, \mu)=f(\bar X),~ \forall \bar X \in \mathbb{R}^{m\times n}$;
	\item [(ii)] (convexity) $\tilde{f}(X, \mu)$ is convex with respect to $X$ for any fixed $\mu>0$;
	\item [(iii)] (gradient consistency) $\left\{\lim\limits _{Z \rightarrow X, \mu \downarrow 0} \nabla_{Z} \tilde{f}(Z, \mu)\right\} \subseteq \partial f(X),~ \forall X \in \mathbb{R}^{m\times n}$;
	\item [(iv)] ( $ \tilde{f}\left(X, \cdot\right) $ Lipschitz continuity with respect to $\mu$) there exists a positive constant $\kappa$ such that
	$$
	\left|\tilde{f}\left(X, \mu_{2}\right)-\tilde{f}\left(X, \mu_{1}\right)\right| \leq \kappa\left|\mu_{1}-\mu_{2}\right|,~\forall X \in \mathbb{R}^{m\times n},~ \mu_{1}, \mu_{2} \in[0, \bar{\mu}];
	$$
	\item [(v)] ($\nabla_{X} \tilde{f}(\cdot, \mu)$ Lipschitz continuity with respect to $X$) there exists a constant $L>0$ such that for any $\mu \in(0, \bar{\mu}]$, $\nabla_{X} \tilde{f}(\cdot, \mu)$ is Lipschitz continuous with Lipschitz constant $L \mu^{-1}$.
\end{itemize}
\end{defi}

Throughout this paper, we denote $\tilde{f}(X,\mu)$ a smoothing function of $f(X)$ in \eqref{lc}. For convenience of notation, the gradient of $\tilde{f}(X, \mu)$ with respect to $X$ is denoted as $\nabla \tilde{f}(X, \mu) $. Furthermore, Definition \ref{def:smoo}-(iv) indicates that
\begin{equation}\label{eq:kapmu}
|\tilde{f}(X, \mu)-f(X)| \leq \kappa \mu, \quad \forall X \in \mathbb{R}^{m\times n},\, 0<\mu \leq \bar{\mu}.	
\end{equation}
Some smoothing functions of the $l_1$ loss function in \eqref{l1-loss} can be found in Example 3.1 of \cite{BC20} and we omit it here.

Some notations are listed here.
$$
\tilde{\mathcal{F}}^{d}(X, \mu) \triangleq \tilde{f}(X, \mu)+\lambda \Phi^{d}(X) \quad \text { and } \quad \tilde{\mathcal{F}}(X, \mu) \triangleq \tilde{f}(X, \mu)+\lambda \Phi(X),
$$
where $\tilde{f}$ is a smoothing function of $f$, $\mu>0$ and $d \in \mathbb{D}^{n}$. For any fixed $\mu>0$ and $d \in \mathbb{D}^{n}$, both $\tilde{\mathcal{F}}^{d}(X, \mu)$ and $\tilde{\mathcal{F}}(X, \mu)$ are nonconvex. Moreover,  $\tilde{\mathcal{F}}^{d}(X, \mu)$ is smooth, but $\tilde{\mathcal{F}}(X, \mu)$ is nonsmooth. Moreover,
\begin{equation}\label{eq:dlg}
	\tilde{\mathcal{F}}^{d}(X, \mu) \geq \tilde{\mathcal{F}}(X, \mu), \quad \forall d \in \mathbb{D}^{n},\, X \in \mathbb{R}^{m\times n},\, \mu \in(0, \bar{\mu}].
\end{equation}

Now we are ready  to recall some preliminaries on proximal gradient method.

Similar to the analysis in Subsection 3.2 of \cite{BC20}, we have a closed-form solution  to proximal operator of $\tau \Phi^{d}$ as follows.
\begin{lem}\label{the:pro-gra-vec}
For any given vectors $d \in \mathbb{D}^{n},\, w \in \mathbb{R}^{n}_+$, and a positive number $\tau>0,$ the proximal operator of $\bm{prox}_{\tau\Phi^d}(w)$ has a closed-form solution; i.e.,
\begin{equation}\label{}
\hat{x}=\bm{prox}_{\tau\Phi^d}(w):=\arg \min _{x \in \mathbb{R}^n_+}\left\{\tau \Phi^{d}(\mathscr D(x))+\frac{1}{2}\|x-w\|_F^{2}\right\}	
\end{equation}
can be calculated by
\begin{equation}\label{eq:y}
\hat{x}_{i}=\max\{\bar{w}_{i}-\tau / \nu,0\}, \quad i\in [n],		
\end{equation}
where
\[\bar w_i=\left\{\begin{array}{ll}w_i &\mbox{if} \,\, d_i=1,\\
w_i+\frac{\tau}{\nu}&\mbox{if} \,\, d_i=2.\end{array}\right.\]
\end{lem}

\begin{theorem}\label{the:pro-gra-m}
For given $W \in \re^{m \times n}$ and $\tau >0$, let $U\mathscr D(w)V^T$ be the singular value decomposition of $W$ and $\hat{x}=\bm{prox}_{\tau\Phi^d}(w)$. Then $ \hat{x}_1 \ge \hat{x}_2 \ge \ldots \ge\hat{x}_n\ge0 $
and $\hat{X}=U\mathscr D(\hat{x}) V^T$ is an optimal solution of the problem
\begin{equation}\label{eq:pro-gra}
 \min _{X}\left\{\tau \Phi^{d}(X)+\frac{1}{2}\|X-W\|_F^{2}\right\}.	
\end{equation}

\end{theorem}
\begin{proof}
From $ w_1\ge w_2\ge \ldots \ge w_n\ge 0$, it is clear that  $ d_1\ge d_2\ge \ldots \ge d_n $. Next, we will prove that $ \hat{x}_1\ge\hat{x}_2\ge\ldots\geq\hat x_n\geq 0 $. We split the proof into three cases.

Case 1. $ d_i= d_{i+1}=2$. By \eqref{eq:y}, it holds
$$ \hat{x}_i=w_{i}\ge w_{i+1}=\hat{x}_{i+1}\ge0. $$

Case 2.  $ d_i=2 $ and $ d_{i+1}=1 $. By \eqref{eq:y}, it holds
$$ \hat{x}_i=w_{i}\ge\max\left\lbrace w_{i+1}-\tau/\nu,0\right\rbrace=\hat{x}_{i+1}\ge0. $$

Case 3.  $ d_i=d_{i+1}=1 $. By \eqref{eq:y}, it holds
$$ \hat{x}_i=\max\left\lbrace w_{i}-\tau/\nu,0\right\rbrace \ge\max\left\lbrace w_{i+1}-\tau/\nu,0\right\rbrace=\hat{x}_{i+1}\ge0. $$
Combining with all cases, the non-increasing of $\hat x_i$ is asserted.

Invoking by \cite[Proposition 2.1]{LZL15} with $F(X)=\tau \Phi^{d}(X),\, \phi(t)=t^{2}/2$, $\|\cdot\|=\|\cdot\|_{F}$ and using the fact that $\hat{x}$ is an optimal solution of $\bm{prox}_{\tau\Phi^d}(w)$, it is concluded that $\hat{X}=U \mathscr D\left(\hat{x}\right) V^{T}$ is an optimal solution of \eqref{eq:pro-gra}.
\end{proof}
Before we end this subsection, we consider the following approximation of $\tilde{\mathcal{F}}^{d}(\cdot, \mu)$ on a given matrix $Z$
\begin{equation}\label{eq:Q}
Q_{d, \gamma}(X, Z, \mu)=\tilde{f}(Z, \mu)+\langle X-Z, \nabla \tilde{f}(Z, \mu)\rangle+\frac{1}{2} \gamma \mu^{-1}\|X-Z\|_F^{2}+\lambda \Phi^{d}(X)	
\end{equation}
with a constant $\gamma>0$.  Then, minimization problem $\min_X Q_{d, \gamma}(X, Z, \mu)$ has a closed form, denoted by $\hat{X}$, which can be calculated by Theorem \ref{the:pro-gra-m} with $\tau=\lambda \gamma^{-1} \mu$ and $W=Z-\gamma^{-1} \mu \nabla \tilde{f}(Z, \mu)$.
\subsection{SPG algorithm}
In this subsection, a proximal gradient algorithm based on the smoothing method, denoted by SPG for simplicity,  will be established for finding a lifted stationary point of \eqref{lc}.

The
following assumptions are needed in the convergence analysis of the SPG algorithm:
\begin{itemize}
	\item (A1) Assumption 1 and Assumption 2 hold;
	\item (A2) $\tilde{f}$ is a smoothing function of $f$ defined in Definition \ref{def:smoo};
	\item (A3) The global minimum point of $\mathcal{F}$ in \eqref{lc} (or $\mathcal{F}_{\ell_{0}}$ in \eqref{l0}) is bounded.
\end{itemize}
Borrowing from  $L_{f}$ in Assumption 1, $\nu$ can be defined such that problems \eqref{l0} and \eqref{lc} have the consistency in Theorem \ref{tho:gm} and Proposition \ref{pro:lm}. Parameter $\kappa$ in Definition \ref{def:smoo} is used in the SPG algorithm, which can be calculated exactly for most smoothing functions \cite{C12}. 
 The value of $L$ in Definition \ref{def:smoo} is not necessary, and we will use a simple line search method to find an acceptable value at each iteration of the SPG algorithm.

 Based on the above assumptions, the SPG algorithm for solving \eqref{lc} is outlined as Algorithm \ref{alg} here.
\begin{algorithm}
	\caption{SPG algorithm.}
	\label{alg}
	\begin{algorithmic}[]
		\REQUIRE Let $X^{0}\in\mathbb R^{m\times n}$ and $\mu_{-1}=\mu_{0} \in(0, \bar{\mu}]$. Choose $\rho,\,\sigma > 1,\, \alpha>0$, and $0<\underline{\gamma} \leq \bar{\gamma}$. Set $k:=0$.
		\WHILE{not converge}
		\STATE \bm{Step~ 1.} Choose $\gamma_{k} \in[\underline{\gamma}, \bar{\gamma}]$ and let $d^{k} := d^{X^{k}}$, where $d^{X^{k}}$ is defined in \eqref{d}.
		\STATE \bm{Step ~2.} 2a) Compute
		\begin{equation}\label{eq:hx}
			\hat{X}^{k+1}=\arg \min_X Q_{d^{k}, \gamma_{k}}\left(X, X^{k}, \mu_{k}\right).
		\end{equation}		
		\STATE \qquad\quad\, 2b) If $\hat{X}^{k+1}$ satisfies
		\begin{equation}\label{eq:flq}
			\tilde{\mathcal{F}}^{d^{k}}\left(\hat{X}^{k+1}, \mu_{k}\right) \leq Q_{d^{k}, \gamma_{k}}\left(\hat{X}^{k+1}, X^{k}, \mu_{k}\right),
		\end{equation}
	let
		\begin{equation}\label{eq:x_k}
			X^{k+1}:=\hat{X}^{k+1}
		\end{equation}
	    and go to \bm{Step ~3}. Otherwise, $\gamma_{k}:=\rho \gamma_{k},$ and return to Step 2a).
		\STATE \bm{Step~ 3.} If
		\begin{equation}\label{ag:mu_if}
			\tilde{\mathcal{F}}\left(X^{k+1}, \mu_{k}\right)+\kappa \mu_{k}-\tilde{\mathcal{F}}\left(X^{k}, \mu_{k-1}\right)-\kappa \mu_{k-1} \leq-\alpha \mu_{k},
		\end{equation}
		set $\mu_{k+1}=\mu_{k} ;$ otherwise, set
		\begin{equation}\label{ag:mu}
			\mu_{k+1}:=\frac{\mu_{0}}{(k+1)^{\sigma}}.
		\end{equation}		
		\STATE Let $ k:=k+1 $ and go to \bm{Step~ 1}.
		\ENDWHILE
		\ENSURE $X^{k+1}$.
	\end{algorithmic}
\end{algorithm}

At each iteration,  the proximal gradient algorithm is adopted for solving
$\min_{X}$ $Q_{d^k, \gamma_k}(X, X^k, \mu_k)$ with fixed $\mu_{k},\, \gamma_{k},$ and $d^{k}$. The values of $\gamma_{k}$ are chosen independently in Step 1 of each iteration. Step 3 updates the smoothing parameter $\mu_{k}$ by \eqref{ag:mu_if}, where $\tilde{\mathcal{F}}\left(X^{k+1}, \mu_{k}\right)+\kappa \mu_{k}$ can be seen as an energy function, with monotone nonincreasing property, which can be seen from Lemma \ref{lem:noc}. If the energy function decreases more than the given scale, then the smoothing parameter $ \mu_k $ is still acceptable; otherwise, we reduce it by the updating rule \eqref{ag:mu}. Let
$$
\mathcal{N}^{s}=\left\{k \in \mathbb{N}: \mu_{k+1} \neq \mu_{k}\right\},
$$
and denote $n_{r}^{s}$ the $r$th smallest number in $\mathcal{N}^{s}$. Then, we can update $\left\{\mu_{k}\right\}$ by
\begin{equation}\label{eq:mk}
\mu_{k}=\mu_{n_{r}^{s}+1}=\frac{\mu_{0}}{\left(n_{r}^{s}+1\right)^{\sigma}}, \quad \forall n_{r}^{s}+1 \leq k \leq n_{r+1}^{s},	
\end{equation}
which will be used in the proof of Lemma \ref{lem:mu}.

\subsection{Convergence analysis}\label{sub:ca}
In this subsection, we will present the convergence analysis for the SPG algorithm.

Let $\left\{X^{k}\right\}$, $\left\{\gamma_{k}\right\}$ and $\left\{\mu_{k}\right\}$ be the sequences generated by the SPG algorithm. We first show that the SPG algorithm is well-defined.  Then we establish some basic properties of the iterates $\left\{X^{k}\right\}$, $\left\{\gamma_{k}\right\}$ and $\left\{\mu_{k}\right\}$ in Lemma \ref{lem:wd}-\ref{lem:pro}. Next, the subsequential convergence of $\left\{X^{k}\right\}$ to a lifted stationary point of \eqref{lc} is established in Proposition \ref{pro:ac}. Finally, we prove the global sequence convergence of iterates $\left\{X^{k}\right\}$ in Theorem \ref{tho:gc}.

\begin{lemma}\label{lem:wd}
	The SPG algorithm is well-defined and 
	$\left\{\gamma_{k}\right\} \subseteq[\underline{\gamma}, \max \{\bar{\gamma}, \rho L\}]$.
\end{lemma}
\begin{proof}
 Clearly, \eqref{eq:flq} holds if and only if
	$$
	\tilde{f}\left(\hat{X}^{k+1}, \mu_{k}\right) \leq \tilde{f}\left(X^{k}, \mu_{k}\right)+\left\langle\nabla \tilde{f}\left(X^{k}, \mu_{k}\right), \hat{X}^{k+1}-X^{k}\right\rangle+\frac{1}{2} \gamma_{k} \mu_{k}^{-1}\left\|\hat{X}^{k+1}-X^{k}\right\|_F^{2}.
	$$
	Invoking Definition \ref{def:smoo}-(v), \eqref{eq:flq} holds when $\gamma_{k} \geq L $. Thus, the updating of $\gamma_{k}$ in Step 2 is at most $\log _{\eta}(L / \underline{\gamma})+1$ times at each iteration. Hence, the SPG algorithm is well-defined, and we have that $\gamma_{k} \leq \max \{\bar{\gamma}, \rho L\},~ \forall k \in \mathbb{N}$.	
\end{proof}

\begin{lemma}\label{lem:noc}
	For any $k \in \mathbb{N}$, we have
	\begin{equation}\label{eq:noc}
	\tilde{\mathcal{F}}\left(X^{k+1}, \mu_{k}\right) \leq \tilde{\mathcal{F}}\left(X^{k}, \mu_{k}\right),	
	\end{equation}
	which implies that $\left\{\tilde{\mathcal{F}}\left(X^{k+1}, \mu_{k}\right)+\kappa \mu_{k}\right\}$ is nonincreasing.
\end{lemma}
\begin{proof}
	By  \eqref{eq:hx}, it follows
	$$
	Q_{d^{k}, \gamma_{k}}\left(X^{k+1}, X^{k}, \mu_{k}\right) \leq Q_{d^{k}, \gamma_{k}}\left(X, X^{k}, \mu_{k}\right), \quad \forall X \in \mathbb{R}^{m\times n}.
	$$
	From \eqref{eq:Q}, upon rearranging the terms, we have
	\begin{equation}\label{3.6-1}
	\begin{aligned}
	\lambda \Phi^{d^{k}}\left(X^{k+1}\right) \leq & \lambda \Phi^{d^{k}}\left(X\right)+\left\langle X-X^{k+1}, \nabla \tilde{f}\left(X^{k}, \mu_{k}\right)\right\rangle +\frac{1}{2} \gamma_{k} \mu_{k}^{-1}\left\|X-X^{k}\right\|_F^{2}\\
	&-\frac{1}{2} \gamma_{k} \mu_{k}^{-1}\left\|X^{k+1}-X^{k}\right\|_F^{2}.
	\end{aligned}	
	\end{equation}
	Moreover, \eqref{eq:flq} can be written as
	\begin{equation}\label{3.6-2}
	\begin{aligned}
	\tilde{\mathcal{F}}^{d^{k}}\left(X^{k+1}, \mu_{k}\right) \leq& \tilde{f}\left(X^{k}, \mu_{k}\right)+\left\langle X^{k+1}-X^{k}, \nabla \tilde{f}\left(X^{k}, \mu_{k}\right)\right\rangle
	\\&	+\frac{1}{2} \gamma_{k} \mu_{k}^{-1}\left\|X^{k+1}-X^{k}\right\|_F^{2}+\lambda \Phi^{d^{k}}\left(X^{k+1}\right).		
	\end{aligned}
	\end{equation}
	Summing up \eqref{3.6-1} and \eqref{3.6-2}, there holds
	\begin{equation}\label{3.6-3}
	\begin{aligned}
	\tilde{\mathcal{F}}^{d^{k}}\left(X^{k+1}, \mu_{k}\right) \leq& \tilde{f}\left(X^{k}, \mu_{k}\right)+\lambda \Phi^{d^{k}}(X)+\left\langle X-X^{k}, \nabla \tilde{f}\left(X^{k}, \mu_{k}\right)\right\rangle \\
	&+\frac{1}{2} \gamma_{k} \mu_{k}^{-1}\left\|X-X^{k}\right\|_F^{2}, \quad \forall X \in \mathbb{R}^{m\times n}.		
	\end{aligned}
	\end{equation}
	For a fixed $\mu>0$, the convexity of $\tilde{f}(X, \mu)$ with respect to $X$ leads to
	\begin{equation}\label{3.6-4}
	\tilde{f}\left(X^{k}, \mu_{k}\right)+\left\langle X-X^{k}, \nabla \tilde{f}\left(X^{k}, \mu_{k}\right)\right\rangle \leq \tilde{f}\left(X, \mu_{k}\right), \quad \forall X \in \mathbb{R}^{m\times n}.	
	\end{equation}
	Combining \eqref{3.6-3} and \eqref{3.6-4} and recalling the definition of $\tilde{\mathcal{F}}^{d^{k}},$ it follows
	\begin{equation}\label{3.6-5}
	\begin{aligned}
	\tilde{\mathcal{F}}^{d^{k}}\left(X^{k+1}, \mu_{k}\right) \leq& \tilde{\mathcal{F}}^{d^{k}}\left(X, \mu_{k}\right)+\frac{1}{2} \gamma_{k} \mu_{k}^{-1}\left\|X-X^{k}\right\|_F^{2}, \quad \forall X \in \mathbb{R}^{m\times n}.			
	\end{aligned}
	\end{equation}
	Letting $X=X^{k}$ in \eqref{3.6-5} and by $d^{k}=d^{X^{k}},$ we obtain $\Phi^{d^k}(X^k)=\Phi(X^k)$, and hence
	\begin{equation}\label{3.6-6}
	\tilde{\mathcal{F}}^{d^{k}}\left(X^{k+1}, \mu_{k}\right) \leq \tilde{\mathcal{F}}\left(X^{k}, \mu_{k}\right).	
	\end{equation}
	Thanks to \eqref{eq:dlg},  $\tilde{\mathcal{F}}^{d^{k}}\left(X^{k+1}, \mu_{k}\right) \geq \tilde{\mathcal{F}}\left(X^{k+1}, \mu_{k}\right)$. Therefore, \eqref{3.6-6} leads to \eqref{eq:noc}.
	
	Since $\tilde{\mathcal{F}}(X,\mu)=\tilde{f}(X,\mu)+\lambda \Phi(X)$, it is  clear that
	\[\tilde{\mathcal{F}}(X,\mu_k)-\tilde{\mathcal{F}}(X,\mu_{k-1})=\tilde
	{f}(X,\mu_k)-\tilde{f}(X,\mu_{k-1})\leq \kappa\left(\mu_{k-1}-\mu_{k}\right),\]
	where the last inequality comes from Definition \ref{def:smoo} (iv). Together with (\ref{eq:noc}), there holds
	\begin{equation}\label{3.6-7}
	\tilde{\mathcal{F}}\left(X^{k+1}, \mu_{k}\right)+\kappa \mu_{k} \leq \tilde{\mathcal{F}}\left(X^{k}, \mu_{k}\right)+\kappa \mu_{k}
	\leq
	\tilde{\mathcal{F}}\left(X^{k}, \mu_{k-1}\right)+\kappa \mu_{k-1},	
	\end{equation}
	which implies the nonincreasing property of $\left\{\tilde{\mathcal{F}}\left(X^{k+1}, \mu_{k}\right)+\kappa \mu_{k}\right\}$.
	\end{proof}

\begin{lemma}\label{lem:mu}
The following statements hold:
\item [(i)] $\sum\limits_{k=0}^{\infty} \mu_{k} \leq \Lambda$ with $\Lambda=\frac{1}{\alpha}\left(\tilde{\mathcal{F}}\left(X^{0}, \mu_{-1}\right)+\kappa \mu_{-1}-\min\limits \mathcal{F}(X)\right)+\frac{ \mu_{0} \sigma}{ \sigma-1}<\infty$;
\item [(ii)] $\lim\limits _{k \rightarrow \infty} \mu_{k}=0$.

\end{lemma}
\begin{proof}
	(i) From \eqref{eq:mk}, we have
	\begin{equation}\label{eq:in}
		\sum_{k \in \mathcal{N}^{s}} \mu_{k}=\sum_{r=1}^{\infty}  \frac{\mu_{0}}{\left(n_{r}^{s}+1\right)^{ \sigma}} \leq \sum_{k=1}^{\infty} \frac{\mu_{0}}{k^{ \sigma}} \leq \frac{ \mu_{0} \sigma}{ \sigma-1},	
	\end{equation}
	where $n_{r}^{s}$ is the $r$th smallest element in $\mathcal{N}^{s}$. By (A3) and \eqref{eq:kapmu}, we see that
	\begin{equation}\label{3.5-2}
		\tilde{\mathcal{F}}\left(X^{k+1}, \mu_{k}\right)+\kappa \mu_{k} \geq \mathcal{F}\left(X^{k+1}\right) \geq \min  \mathcal{F}(X)=\min \mathcal{F}_{\ell_{0}}(X)>-\infty,	
	\end{equation}
	where the equality follows from Theorem \ref{tho:gm}. When $k \notin \mathcal{N}^{s}$, \eqref{ag:mu_if} can be rewritten as \[\alpha \mu_{k} \leq \tilde{\mathcal{F}}\left(X^{k}, \mu_{k-1}\right)+\kappa \mu_{k-1}-\tilde{\mathcal{F}}\left(X^{k+1}, \mu_{k}\right)-\kappa \mu_{k},\]
	which together with the nonincreasing property of $\left\{\tilde{\mathcal{F}}\left(X^{k+1}, \mu_{k}\right)+\kappa \mu_{k}\right\}$ and \eqref{3.5-2} implies that
	\begin{equation}\label{eq:notin}
		\sum_{k \notin \mathcal{N}^{s}} \mu_{k} \leq \frac{1}{\alpha}\left(\tilde{\mathcal{F}}\left(X^{0}, \mu_{-1}\right)+\kappa \mu_{-1}-\min \mathcal{F}(X)\right).	
	\end{equation}
	Combining \eqref{eq:in} and \eqref{eq:notin}, the proof for the estimation in item (i) is completed.

	(ii) From (i),  (ii) is obvious. 
\end{proof}

\begin{lemma}\label{lem:pro}
Suppose that the sequence $\left\{X^{k}\right\}$ generated by SPG algorithm  is bounded. Then there exists $ K\in \mathbb{N} $  such that for all $ k \geq  K $, it holds that
\begin{itemize}
	
	\item [(i)] $\left\|\nabla \tilde{f}\left(X^{k}, \mu_{k}\right)\right\|_F<\frac{1}{2}\left(\lambda / \nu+L_{f}\right)$;
	\item [(ii)] $\left\|X^{k+1}-X^{k}\right\|_F \leq \left( {\sqrt n+1} \right)(\lambda /\nu)\gamma _k^{ - 1}{\mu _k}$;
	\item [(iii)] $ \sum\limits_{k = 0}^\infty  {\left\| {{X^{k + 1}} - {X^k}} \right\|_F < \infty }  $.
\end{itemize}
\end{lemma}
\begin{proof}
(i) We argue it by contradiction. Suppose that there exists a subsequence of $\left\{X^{k}\right\}$, denoted by $\left\{X^{k_{i}}\right\}$, such that
\begin{equation}\label{3.10-1}
\left\|\nabla \tilde{f}\left(X^{k_{i}}, \mu_{k_{i}}\right)\right\|_F \geq \frac{1}{2}\left(\lambda / \nu+L_{f}\right)>L_{f}, \quad \forall i \in \mathbb{N}.	
\end{equation}
Since $\left\{X^{k_{i}}\right\}$ is bounded, there exists a subsequence of $\left\{X^{k_{i}}\right\}$ (also denoted by $\left\{X^{k_{i}}\right\}$ for simplicity) and $\bar{X}$ such that $\lim\limits_{i \rightarrow \infty} X^{k_{i}}=\bar{X}$. Due to $\lim\limits_{i \rightarrow \infty} \mu_{k_{i}}=0$, the property of $\tilde{f}$ in Definition \ref{def:smoo}-(iii) and \eqref{3.10-1} imply the existence of $\bar{\xi} \in \partial f(\bar{X})$ such that $\|\bar{\xi}\|_F>L_{f}$, which leads to a contradiction to the definition of $L_{f}$ given in Assumption 1. Hence,  (i) is established.

(ii) Let $W^{k}=X^{k}-\gamma_{k}^{-1} \mu_{k} \nabla\tilde{f}\left(X^{k}, \mu_{k}\right)$ and $U^{k} \mathscr{D}\left(w^{k}\right)\left(V^{k}\right)^{T}$ be the singular value decomposition of $W^{k}$. By \eqref{eq:y}, we have
\begin{equation}\label{eq:x-w-x}
{\left\| {{X^{k + 1}} - {W^k}} \right\|_F} = {\left\| {{x^{k + 1}} - {w^k}} \right\|_F} \le \sqrt n (\lambda / \nu) \gamma_{k}^{-1} \mu_{k}.
\end{equation}
From (i), there exists $ K\in \mathbb{N} $ such that for all $ k>K $, it holds that
\begin{equation}\label{eq:w-x}
\left\|W^k-X^k\right\|_F=\left\|-\gamma_{k}^{-1} \mu_{k} \nabla\tilde{f}\left(X^{k}, \mu_{k}\right)\right\|_F\le (\lambda / \nu)\gamma_{k}^{-1} \mu_{k}.
\end{equation}
Combining \eqref{eq:w-x} and \eqref{eq:x-w-x}, we have
\[{\left\| {{X^{k + 1}} - {X^k}} \right\|_F} \le {\left\| {{X^{k + 1}} - {W^k}} \right\|_F} + {\left\| {{W^k} - {X^k}} \right\|_F} \le \left( {\sqrt n+1} \right)(\lambda /\nu)\gamma _k^{ - 1}{\mu _k},\]
which completes the proof for item (ii).

(ii) From Lemma \ref{lem:mu}-(i) and (ii) of this lemma, we have
\begin{align*}
\sum\limits_{k = 0}^\infty  \left\| {{X^{k + 1}} - {X^k}} \right\|_F &\le \sum\limits_{k = 0}^{K - 1} {\left\| {{X^{k + 1}} - {X^k}} \right\|_F + \sum\limits_{k = K}^\infty  {\left\| {{X^{k + 1}} - {X^k}} \right\|_F} }  \\
&\le \sum\limits_{k = 0}^{K - 1} {\left\| {{X^{k + 1}} - {X^k}} \right\|_F + } {\left( {\sqrt n  + 1} \right)}{(\lambda /\nu )}{\underline\gamma ^{ - 1}}\sum\limits_{k = K}^\infty  {\mu _k}  < \infty.
\end{align*}	
\end{proof}
\begin{pro}\label{pro:ac} Suppose that the  sequence $\left\{X^{k}\right\}$  generated by SPG algorithm is bounded. Then
any accumulation point of $\left\{X^{k}\right\}$ is a lifted stationary point of \eqref{lc}.
\end{pro}
\begin{proof}
Suppose that $\bar{X}$ is an accumulation  point of any convergence subsequent $\left\{X^{k_{i}}\right\}$. By Lemma \ref{lem:pro}-(iii), we have
\begin{equation}
\sum\limits_{i = 0}^\infty  {\left\| {{X^{{k_i} + 1}} - {X^{{k_i}}}} \right\|_F}  \le \sum\limits_{k = 0}^\infty  {\left\| {{X^{k + 1}} - {X^k}} \right\|_F}  < \infty,
\end{equation}
which
implies that
\begin{equation}\label{eq:ac}
\lim _{i \rightarrow \infty}\left\|X^{k_{i}+1}-X^{k_{i}}\right\|_F=0 \text{ and } \lim _{i \rightarrow \infty} X^{k_{i}+1}=\bar{X}.	
\end{equation}

Recalling $X^{k_{i}+1}=\hat{X}^{k_{i}+1}$ defined in \eqref{eq:hx} and by first-order optimality condition, we have
\begin{equation}\label{eq:fo}
\nabla \tilde{f}\left(X^{k_{i}}, \mu_{k_{i}}\right)+\gamma_{k_{i}} \mu_{k_{i}}^{-1}\left(X^{k_{i}+1}-X^{k_{i}}\right)+\lambda \zeta^{k_{i}}= 0, \forall \zeta^{k_{i}} \in \partial \Phi^{d^{k_{i}}}\left(X^{k_{i}+1}\right).
\end{equation}

Since the elements in $\left\{d^{k_{i}}: i \in \mathbb{N}\right\}$ are finite and $\lim\limits _{i \rightarrow \infty} X^{k_{i}+1}=\bar{X}$, there exists a subsequence of $\left\{k_{i}\right\}$, denoted as $\left\{k_{i_{j}}\right\}$, and $\bar{d} \in \mathcal{D}(\sx)$ such that $d^{k_{i_{j}}}=\bar{d},\, \forall j \in \mathbb{N}$. By
the definition of $\partial \Phi^{\bar{d}}$ and $\lim\limits_{j \rightarrow \infty} X^{k_{i_{j}}+1}=\bar{X}$, it gives
\begin{equation}\label{eq:sem}
\left\{\lim_{j \rightarrow \infty} \zeta^{k_{i_{j}}}: \zeta^{k_{i_{j}}} \in \partial \Phi^{d^{k_{i_{j}}}}\left(X^{k_{i_{j}}+1}\right)\right\} \subseteq \partial \Phi^{\bar{d}}(\bar{X}).	
\end{equation}
Along with the subsequence $\left\{k_{i_{j}}\right\}$ and letting $j \rightarrow \infty$ in \eqref{eq:fo}, from Definition \ref{def:smoo}-(iii), \eqref{eq:ac} and \eqref{eq:sem}, we obtain that there exist $\bar{\xi} \in \partial f(\bar{X})$ and $\bar{\zeta}^{\bar{d}} \in \partial \Phi^{\bar{d}}(\bar X)$ such that
\begin{equation}\label{eq:st}
\bar{\xi}+\lambda \bar{\zeta}^{\bar{d}} = 0.	
\end{equation}
By $\bar{d} \in \mathcal{D}(\sx)$ and the definition of $\Phi^{\bar{d}}$ in \eqref{eq:Phi_d}, \eqref{eq:st} implies that $\bar{X}$ is a lifted stationary point of \eqref{lc}.
\end{proof}


\begin{theorem}\label{tho:gc}
	Suppose that the sequence $\left\{X^{k}\right\}$ generated by SPG algorithm is bounded. Then $\left\{X^{k}\right\}$  is globally convergent to a lifted stationary point of \eqref{lc}, i.e., there exists a lifted stationary point $\bar{X}$ of \eqref{lc} such that $\lim\limits_{k \rightarrow \infty} X^{k}=\bar{X}$.
\end{theorem}
\begin{proof}
Suppose that $\left\{X^{k_{j}}\right\}$ is a convergent subsequence of $\left\{X^{k}\right\}$ with
	$\lim\limits _{j \rightarrow \infty} X^{k_{j}}=\bar{X}.	$
	By Proposition \ref{pro:ac}, $\bar{X}$ is a lifted stationary point of \eqref{lc}.
	
	For any $t, s \in \mathbb{N},$ we have
	\begin{equation}\label{3.12-4}
	\left\|X^{t+s+1}-\bar{X}\right\|_F \leq\left\|X^{t}-\bar{X}\right\|_F+\sum_{k=t}^{t+s}\left\|X^{k+1}-X^{k}\right\|_F.	
	\end{equation}
	Foe any given $\epsilon>0$, there exists $K_{1} >0$ such that
	\begin{equation}\label{3.12-5}
	\left\|X^{k_{j}}-\bar{X}\right\|_F \leq \epsilon / 2, \quad \sum_{k=k_{j}}^{\infty}\left\|X^{k+1}-X^{k}\right\|_F \leq \epsilon /2,\quad\forall k_{j} \geq K_{1}.
	\end{equation}
	Here the first inequality dues to $\lim\limits _{j \rightarrow \infty} X^{k_{j}}=\bar{X}$ and the second inequality comes from Lemma \ref{lem:pro}-(iii).
	
	By letting $t=\bar k_{j}\geq K_1$ in \eqref{3.12-4} and  from \eqref{3.12-5}, we obtain $\left\|X^{k}-\bar{X}\right\|_F \leq \epsilon,\, \forall k >$ $K_{1}+1$. From the arbitrariness of $\epsilon>0$,  $\lim\limits_{k \rightarrow \infty} X^{k}=\bar{X}$ follows.
\end{proof}

\section{Numerical experiments}

In this section we conduct numerical experiments to test the performance of the SPG method.
In particular, we apply it to solve the problem \eqref{l0} with $ f(X)=\|P_\Omega\left(X-M \right) \|_1 $, that is,
\begin{equation}\label{NS}
\min_{X \in \re^{m \times n}} {\F_{l_0}(X):=\|P_\Omega\left(X-M \right) \|_1+\lambda\|\sigma(X)\|_0}.
\end{equation}

We conduct extensive experiments to evaluate our method and then comparing it with some existing methods, including FPCA \cite{MGC11}, SVT \cite{CCS08} and VBMFL1 \cite{ZMXZY15}. The platform is Matlab R2014a under Windows 10 on a desktop of a 3.2GHz CPU and 8GB memory. We adopt the root-mean-square error (RMSE) as evaluation metrics
\begin{align*}
RMSE:=\sqrt{\frac{\|X^*-M\|_{F}^2}{mn}},
\end{align*}
and the final performance of each simulation is evaluated by obtaining an ensemble average of the relative error with $ T $ independent Monte Carlo runs.

In the simulation, a typical two-component Gaussian mixture model (GMM) is used as the non-Gaussian noise model. The probability density function (PDF) of GMM is defined as
$$
p_{v}(i)=(1-c) N\left(0, \sigma_{A}^{2}\right)+c N\left(0, \sigma_{B}^{2}\right),
$$
where $N\left(0, \sigma_{A}^{2}\right)$ represents general the noise disturbance with variance $\sigma_{A}^{2}$, and $N\left(0, \sigma_{B}^{2}\right)$ stands for outliers that occur occasionally with a large variance $\sigma_{B}^{2}$. The variable $c$ controls the occurrence probability of outliers.

\subsection{Random Matrix Completion}

In this subsection, we aim to recover a random matrix $M \in \R^{m \times n}$ with rank $ r $ based on a subset of entries $\left\{M_{ij}\right\}_{(i,j) \in \Omega}$. In detail, we first generate random matrices $ {M_L} = unifrnd(-0.1,0.3,m,r) \in {R^{m \times r}} $ and $ {M_R} = unifrnd(-0.1,0.3,n,r) \in {R^{n \times r}} $, then let $M=M_{L} M_{R}^{T}$. We then sample a subset with sampling ratio $SR$ uniformly at random, where $SR=|\Omega| /(m n)$. In our experiment, we set $ m = n$. The GMM noise are set at $\sigma_{A}^{2}=0.0001, \sigma_{B}^{2}=0.1, c=0.1$. The rank $ r $ is set to 30 and the sampling ratio $SR$ is set to 0.8. For each simulation, an average relative error is obtained via 100 Monte Carlo runs with different realizations of $M,\,\Omega$ and noise.

The performance is firstly compared for different $ \mu_0 $ under different $ \mu_k $ iterative methods in step 3 of SPG algorithm. We compare different $ \mu_0 $ under $ \alpha = 0.8 $ and $ \alpha = +\infty $ \footnote{$ \alpha = +\infty $ means $ \mathcal{N}^{s}=\{1,2,\ldots\} $.}. $ \mu_0 $ increases from 10 to 100 with increment 10 and the size of the square matrix $ m $ is set to $ 150 $. 
From Figure \ref{fig:mu}, we can see that the larger $ \mu_0 $ becomes, the smaller the value of RMSE, but the more time and iteration steps it costs. It can also be observed that $ \alpha=0.8 $ has the better performance than $ \alpha = \infty $. This is because when $ \alpha=0.8 $, it is necessary to reduce $ \mu_k $
when \eqref{ag:mu_if} is not satisfied. When $ \alpha = \infty $, no matter how much $ \tilde{\mathcal{F}}\left(X^{k+1}, \mu_{k}\right)+\kappa \mu_{k} $ decreases, $ \mu_k $ will be reduced. It can be seen that the SPG algorithm can accelerate the convergence speed by adjusting the strategy of $ \mu_k $ in step 3.

\begin{figure}[htbp]
	\centering
	\begin{subfigure}[b]{1\linewidth}
		\begin{subfigure}[b]{0.32\linewidth}
			\centering
			\includegraphics[width=\linewidth]{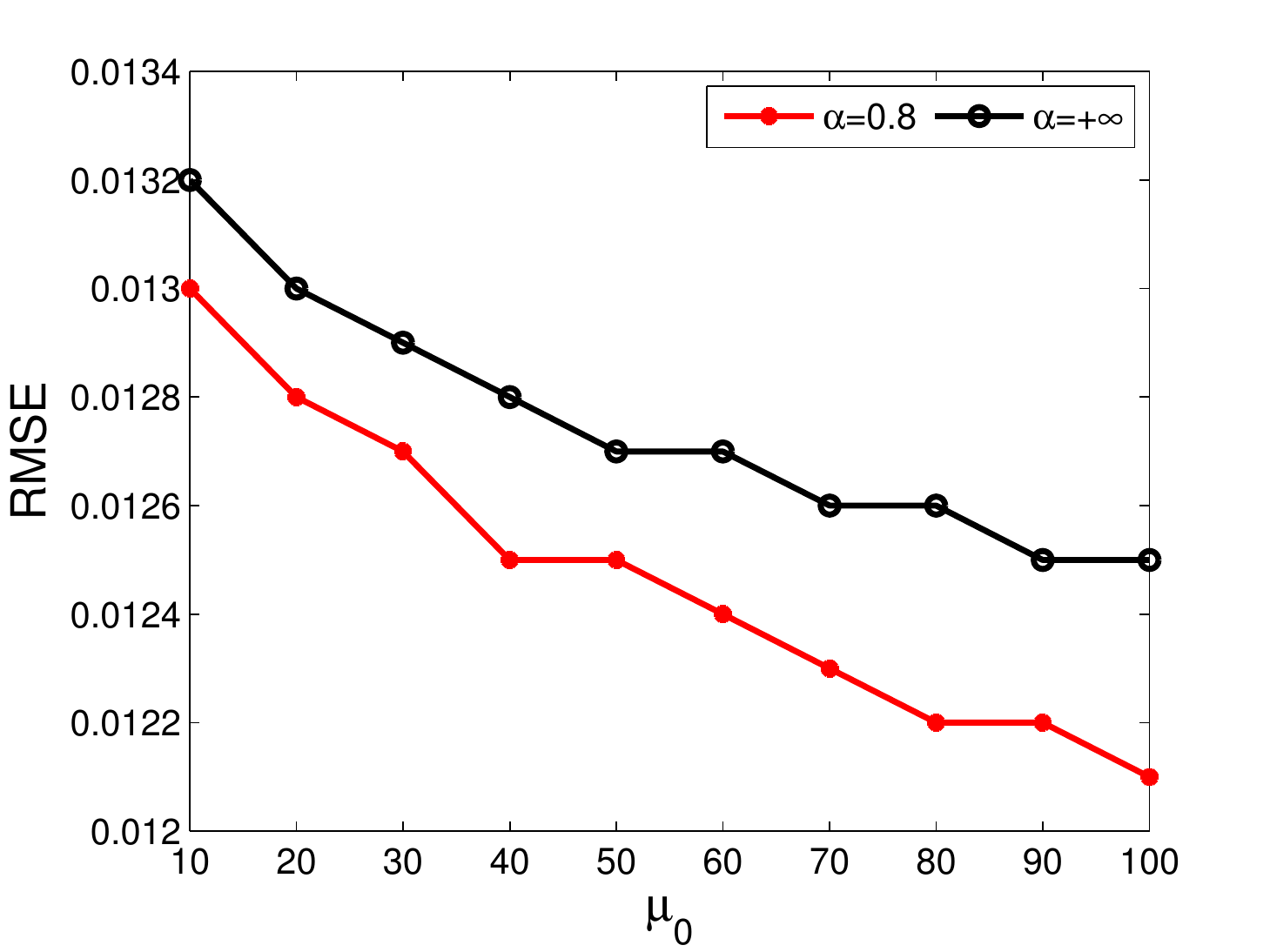}
		\end{subfigure}  	
		\begin{subfigure}[b]{0.32\linewidth}
			\centering
			\includegraphics[width=\linewidth]{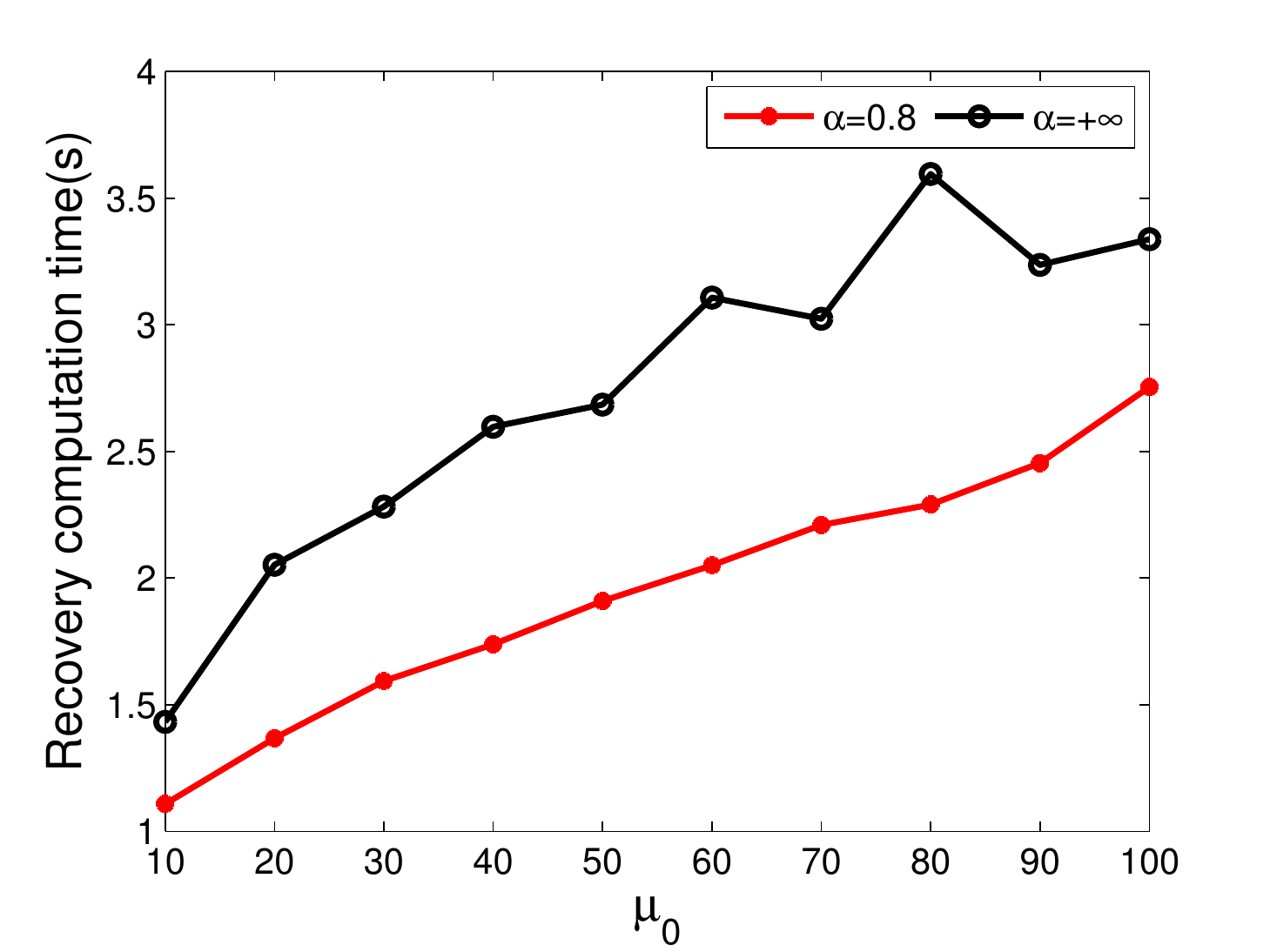}
		\end{subfigure}
			\begin{subfigure}[b]{0.32\linewidth}
		\centering
		\includegraphics[width=\linewidth]{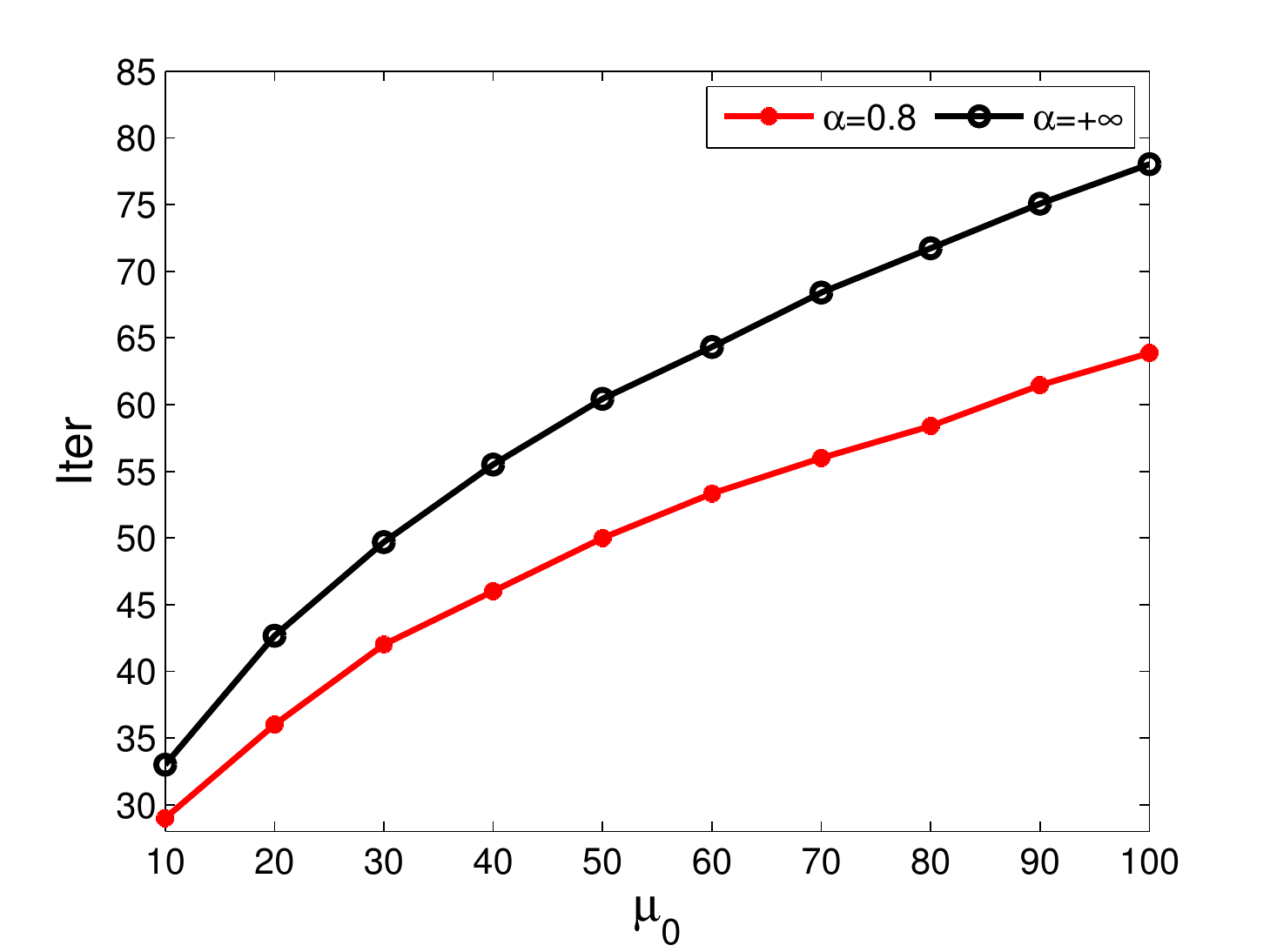}
	\end{subfigure}
	\end{subfigure}
	\vfill
	\caption{Curves of RMSE, average running times and average number of running iterations with different $ \mu_0 $.}
	\label{fig:mu}
\end{figure}

Secondly, the performance of the algorithms for different sizes of completion problems. The size of the square matrix $ m $ increases from 100 to 200 with increment 10. 
Figure \ref{fig:size} shows the curves of the average RMSE and running times in terms of different matrix sizes $ m $. As can be seen from Figure \ref{fig:size}, the SPG algorithm achieves comparably lower
average RMSE than the other algorithms, while FPCA and SVT algorithms based on $ l_2 $ norm have higher RMSE values. Moreover, as
the size of the matrix increases, the average RMSE values decrease for all algorithms. With the increase of matrix size, the average running time of all algorithms increases gradually, but the average running time of SPG algorithm is the least. The average running time of VBMFL1 algorithm based on $ l_1 $ norm increases much faster than the other three algorithms. In summary, SPG algorithm performs best in four algorithms.

\begin{figure}[htbp]
	\centering
	\begin{subfigure}[b]{1\linewidth}
		\begin{subfigure}[b]{0.48\linewidth}
			\centering
			\includegraphics[width=\linewidth]{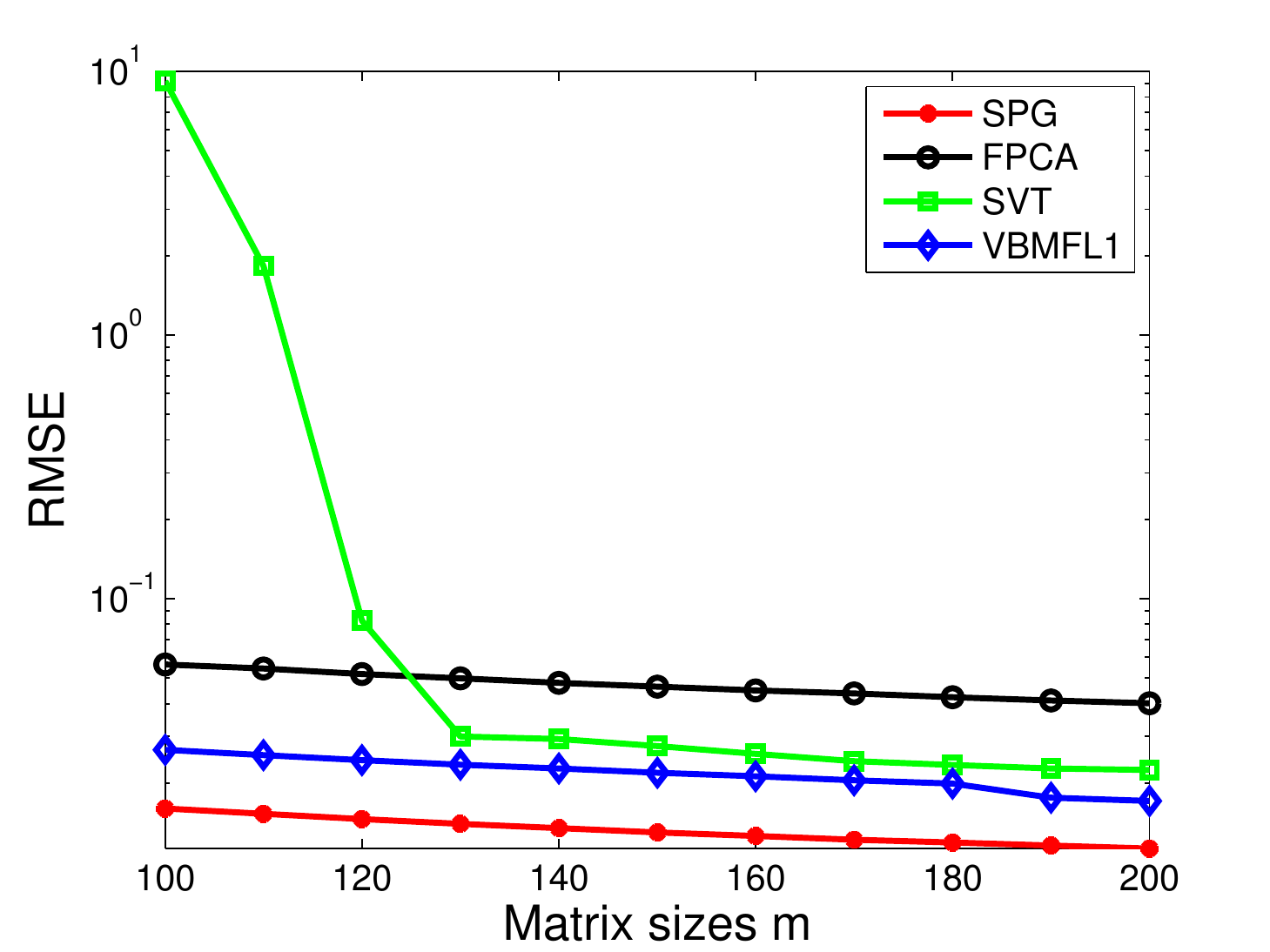}
		\end{subfigure}  	
		\begin{subfigure}[b]{0.48\linewidth}
			\centering
			\includegraphics[width=\linewidth]{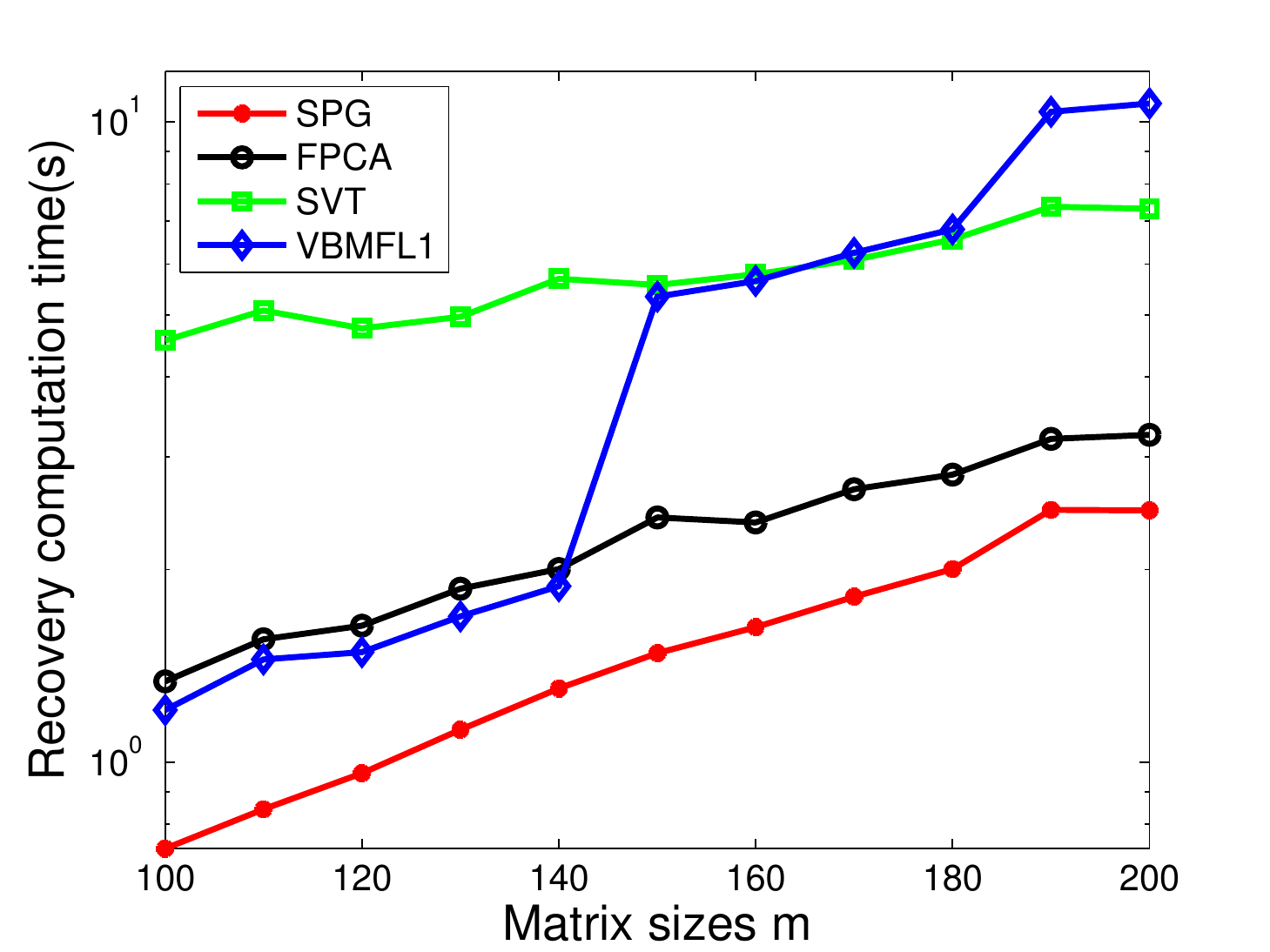}
		\end{subfigure}
	\end{subfigure}
	\vfill
	\caption{Curves of RMSE and average running times with different matrix size $ m $.}
	\label{fig:size}
\end{figure}


\subsection{Image Inpainting}
In this subsection, the performance of the algorithms is compared for some image inpainting tasks with non-Gaussian noise. Note that grayscale images can be expressed as matrix. When the matrix data is of low-rank, or numerical low-rank, the image inpainting problem can be modeled as matrix completion problem. To evaluate the algorithm performances under
non-Gaussian noise, a mixture of Gaussian is selected for the noise model. We adopt the peak signal-to-noise ratio (PSNR) as evaluation metrics, which is defined by
$$
\mathrm{PSNR}:=10 \log _{10}\left(\frac{mn}{\|X^*-M\|_{F}^{2}}\right).
$$
A higher PSNR represents better recovery performance.

We use the USC-SIPI image database\footnote{http://sipi.usc.edu/database/} to evaluate our method for image inpainting. In our test, we randomly select $ 6 $ images from this database for testing and the images are normalized in the range $\left[0,1\right]$. 
In Figure \ref{fig:images-mix-gau}, we consider the case where entries are missing at random by sampling ratio $ SR = 0.9 $. The GMM noise are set at $\sigma_{A}^{2}=0.001, \sigma_{B}^{2}=0.1, c=0.1$. From Figure \ref{fig:images-mix-gau}, we can see that the image restored by FPCA and SVT algorithm with $ l_2 $ norm is very blurred, while the image restored by SPG and VBMFL1 algorithm with $ l_1 $ norm is relatively clear, indicating that the recovery effect of $ l_1 $ norm for non-Gaussian noise is better than that of $ l_2 $ norm.
In the image restored by VBMFL1 algorithm, the recovery effect is not good for those isolated small pixels, especially in ``Chart". It may be that these abnormal small pixels are treated as outliers. The image restored by SPG algorithm performs well in all aspects. At the same time, in order to compare the recovery effect of the five algorithms more clearly, we give the PSNR and running time of the four algorithms in Table \ref{tab:image-gmm}. It can be seen from the table that the  VBMFL1 algorithm based on $ l_1 $ norm have higher PSNR values than the FPCA and SVT algorithms based on $ l_2 $ norm, but the running time is much longer. SPG algorithm has the highest PSNR value and short running time.  We can assert that SPG is the best of the four algorithms.

\begin{figure}[htbp]
	\centering
	\begin{subfigure}[b]{1\linewidth}
		\begin{subfigure}[b]{0.161\linewidth}
			\centering
			\includegraphics[width=\linewidth]{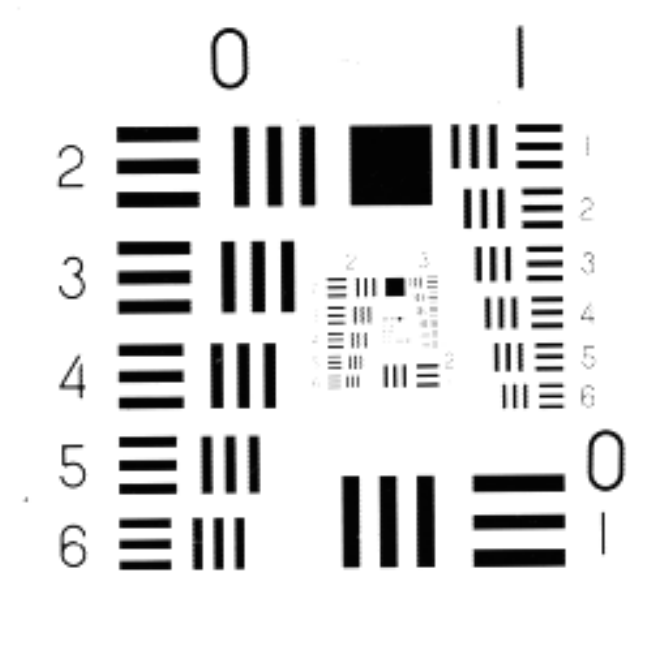}\vspace{0pt}
			\includegraphics[width=\linewidth]{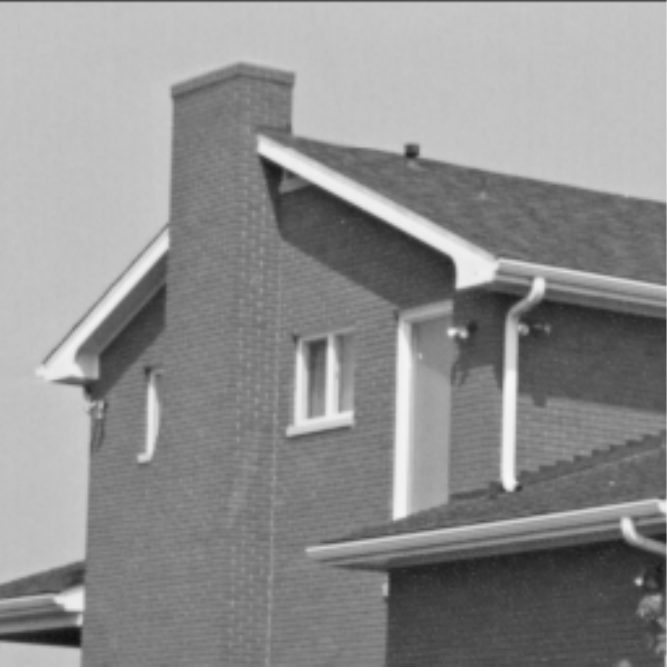}\vspace{0pt}
			\includegraphics[width=\linewidth]{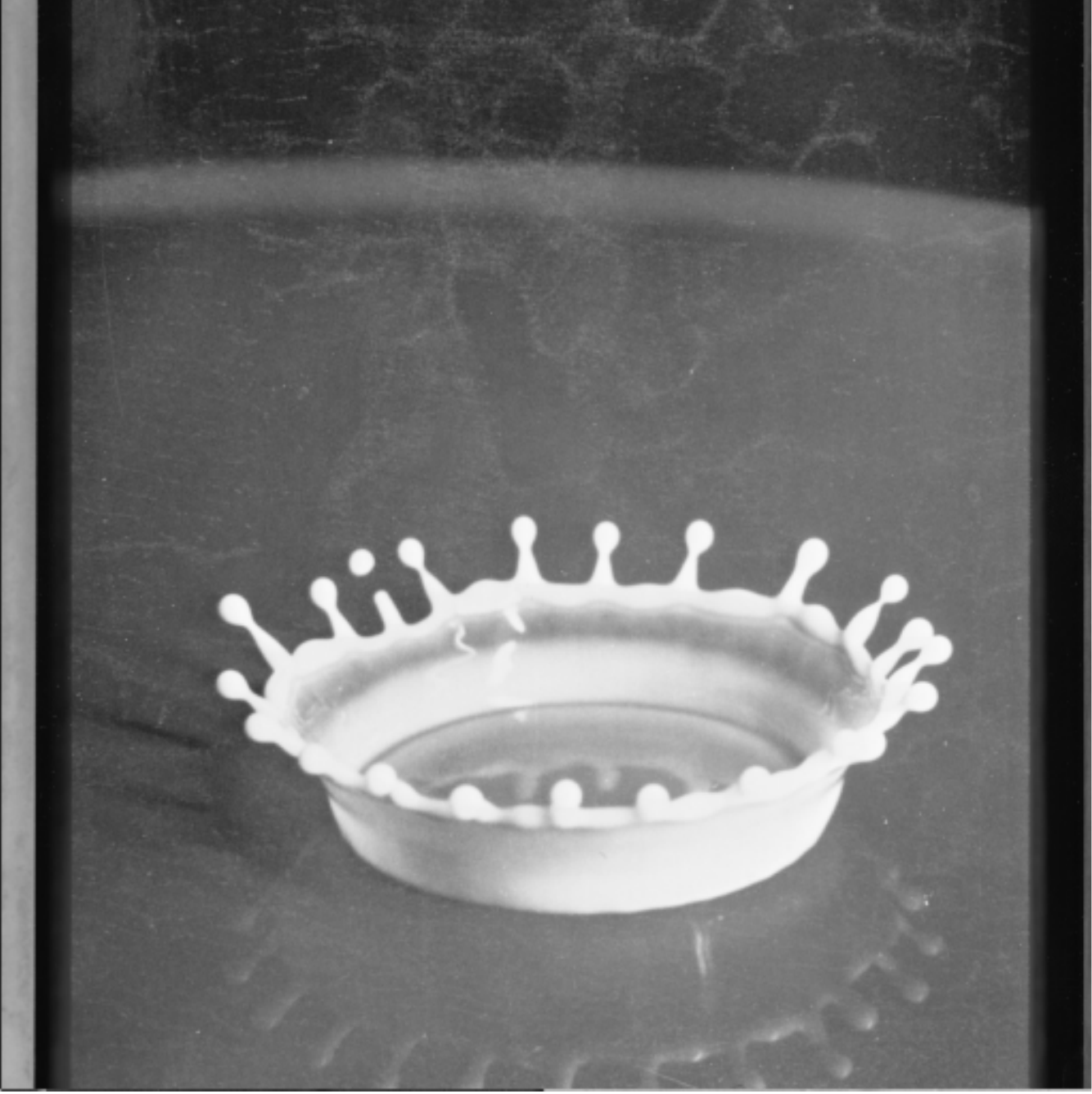}
			\caption{\tiny{Original}}
		\end{subfigure}
		\begin{subfigure}[b]{0.161\linewidth}
			\centering
			\includegraphics[width=\linewidth]{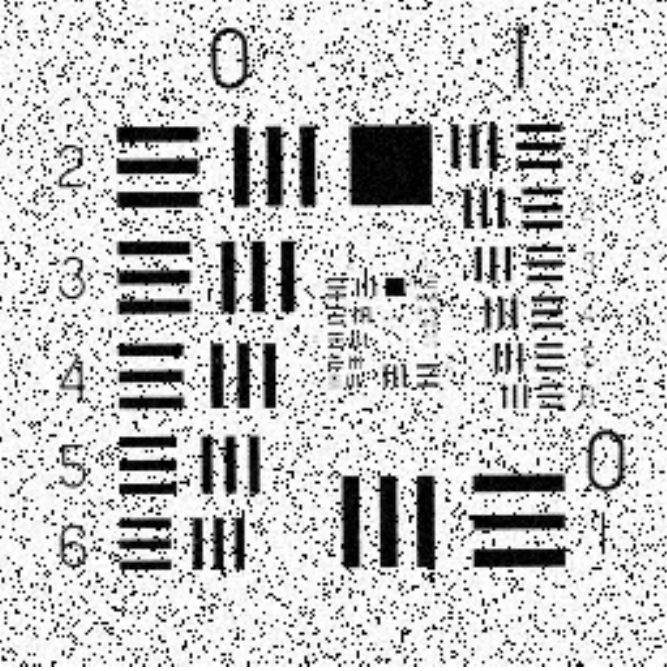}\vspace{0pt}
			\includegraphics[width=\linewidth]{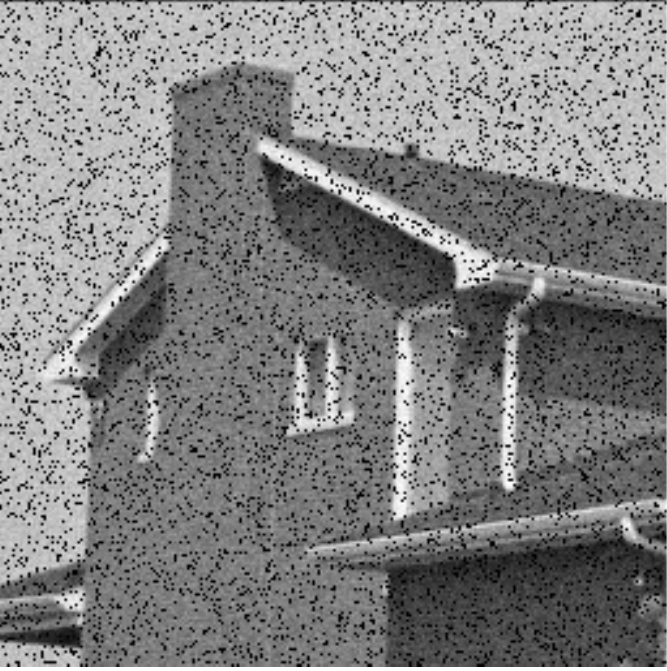}\vspace{0pt}
			\includegraphics[width=\linewidth]{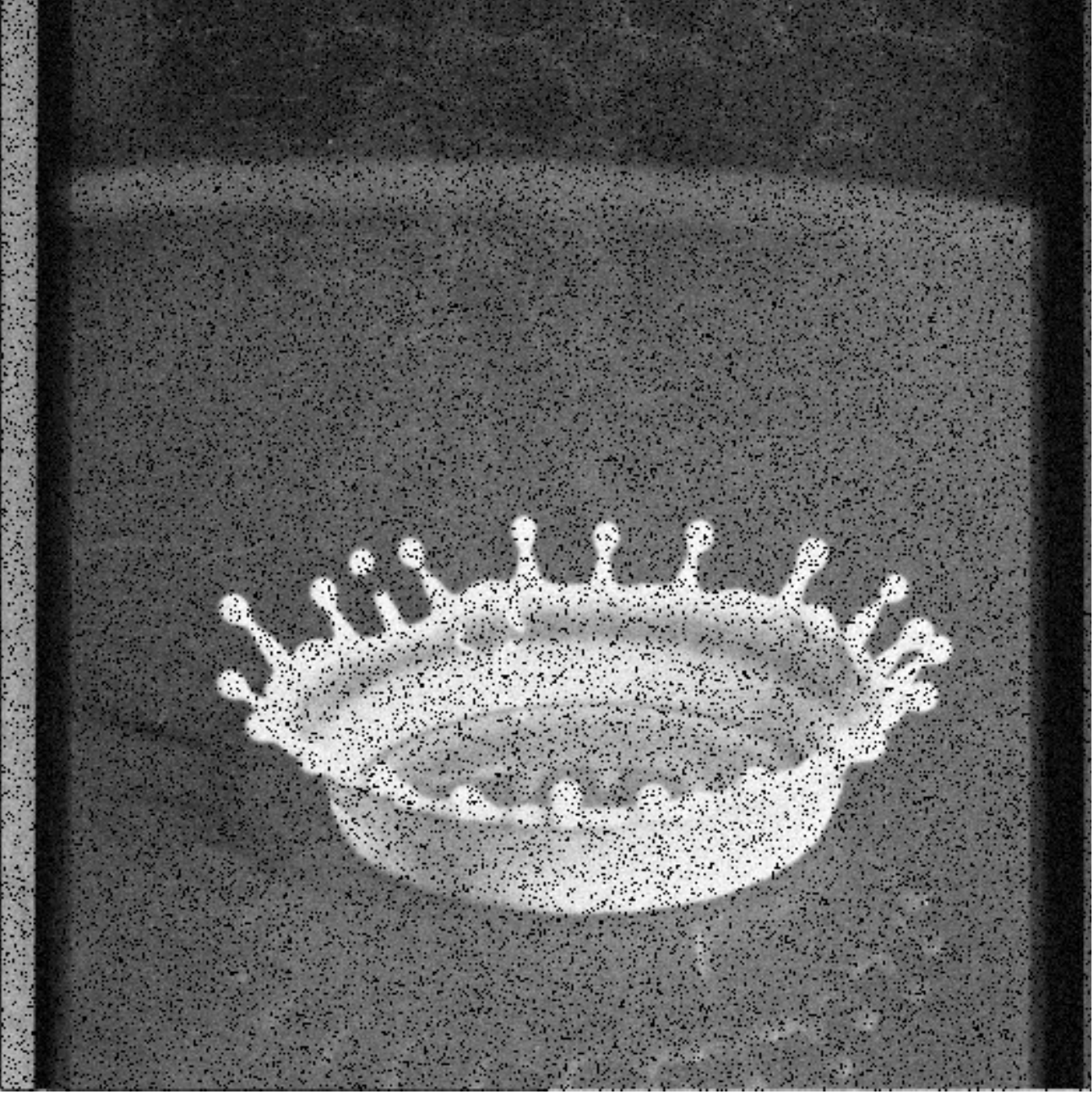}
			\caption{\tiny{Observation}}
		\end{subfigure}
		\begin{subfigure}[b]{0.161\linewidth}
			\centering
			\includegraphics[width=\linewidth]{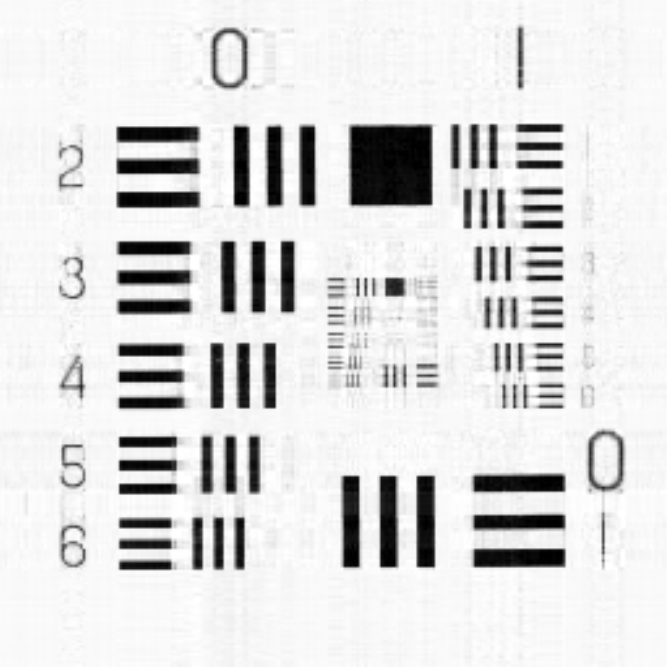}\vspace{0pt}
			\includegraphics[width=\linewidth]{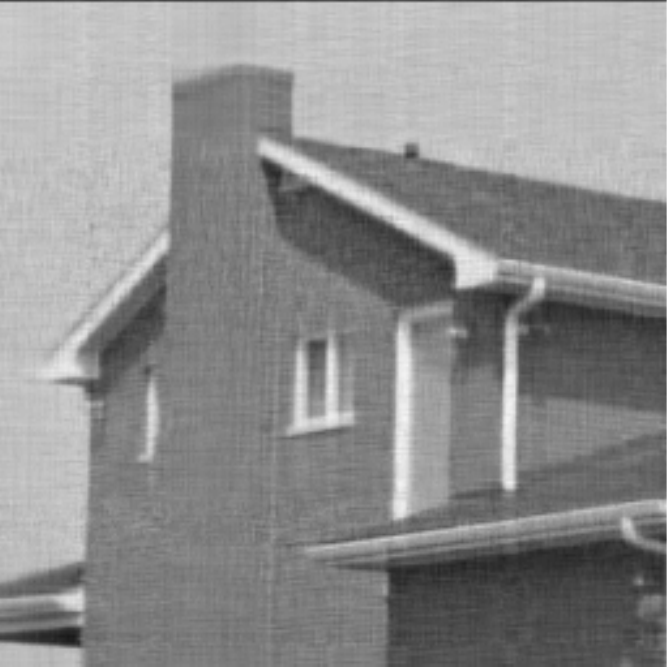}\vspace{0pt}
			\includegraphics[width=\linewidth]{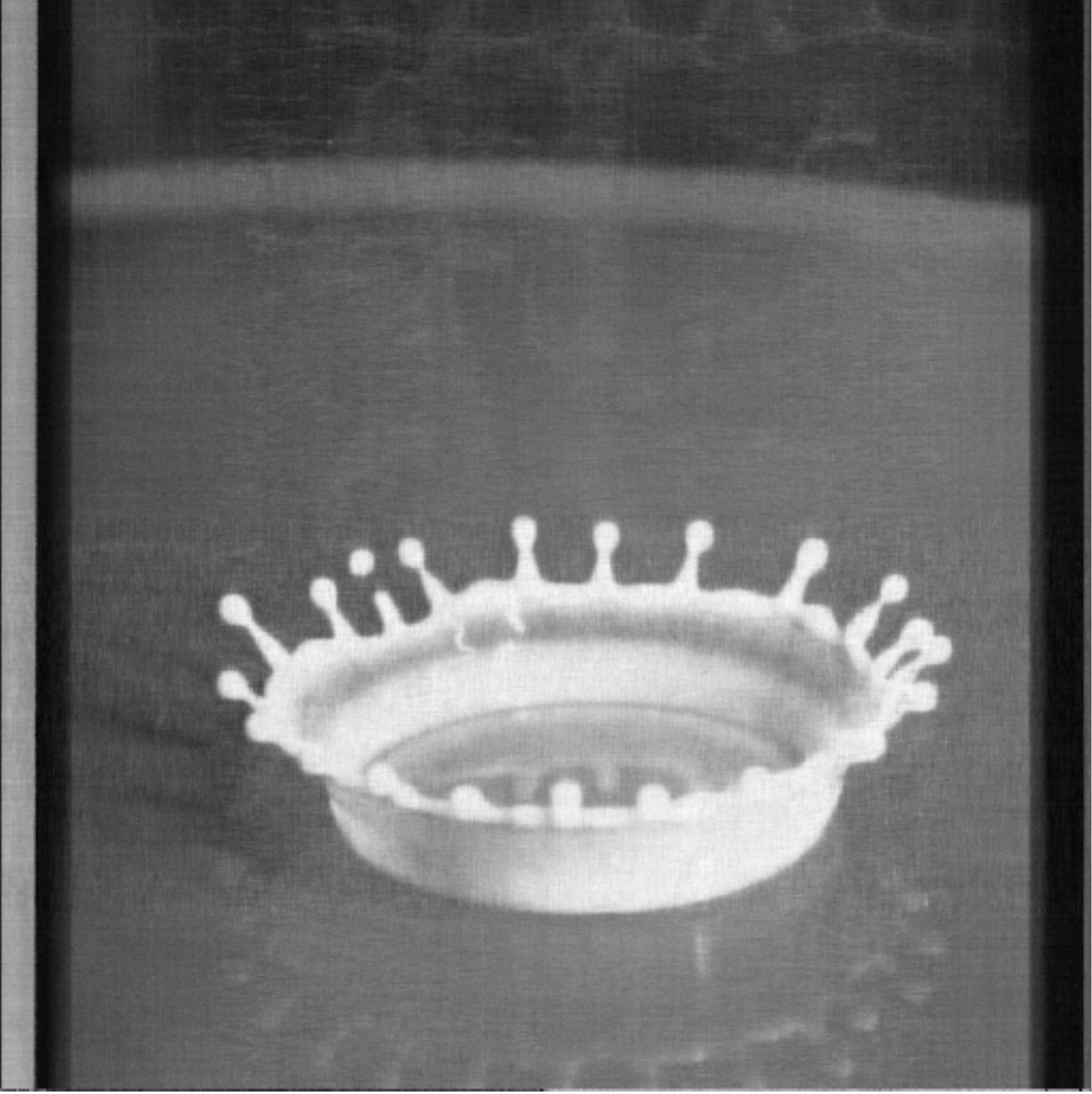}
			\caption{\tiny{SPG}}
		\end{subfigure}
		\begin{subfigure}[b]{0.161\linewidth}
			\centering
			\includegraphics[width=\linewidth]{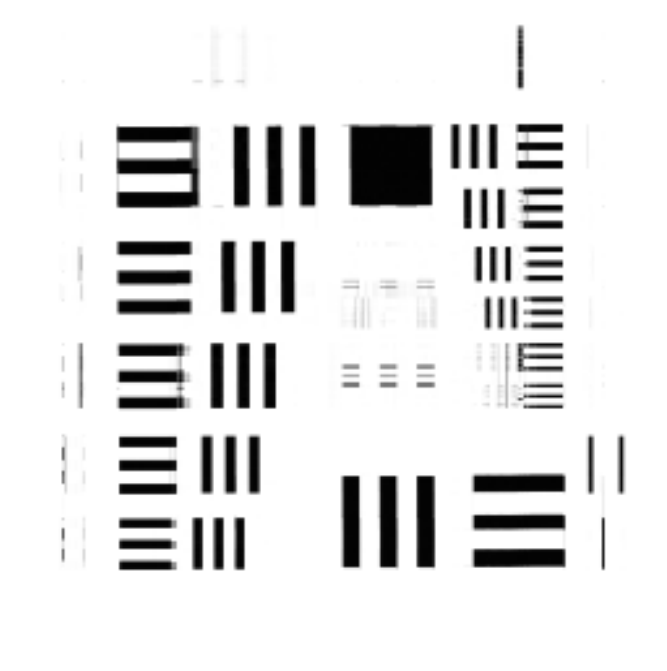}\vspace{0pt}
			\includegraphics[width=\linewidth]{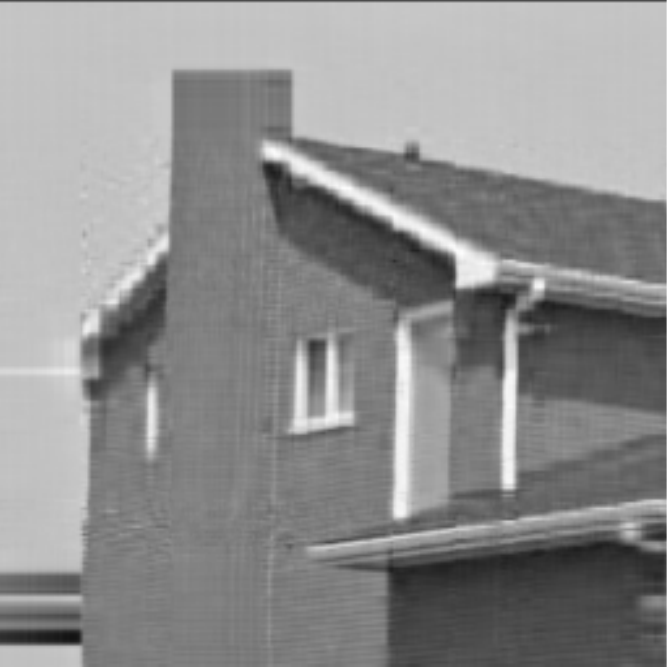}\vspace{0pt}
			\includegraphics[width=\linewidth]{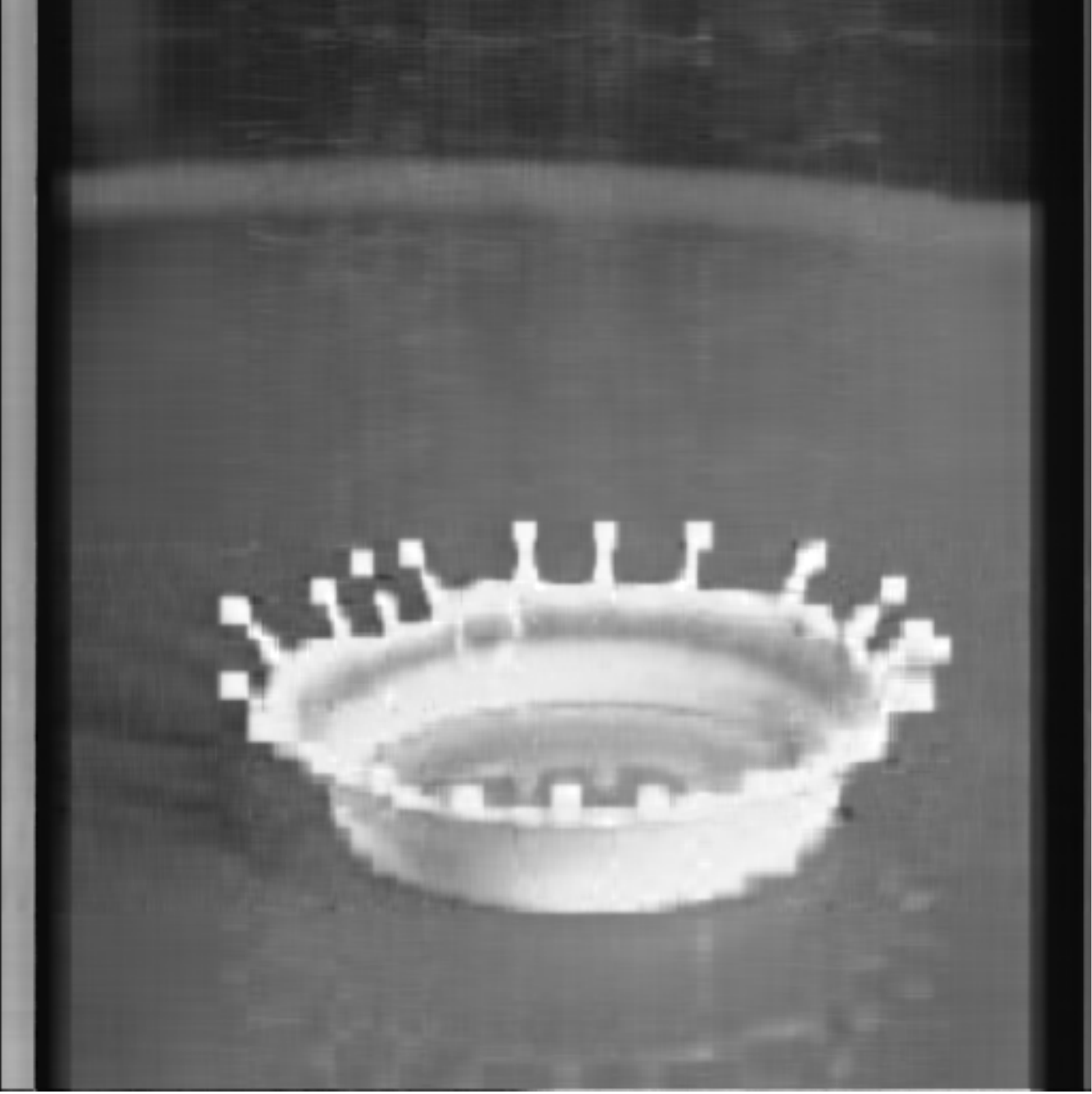}
			\caption{\tiny{VBMFL1}}
		\end{subfigure}
		\begin{subfigure}[b]{0.161\linewidth}
			\centering
			\includegraphics[width=\linewidth]{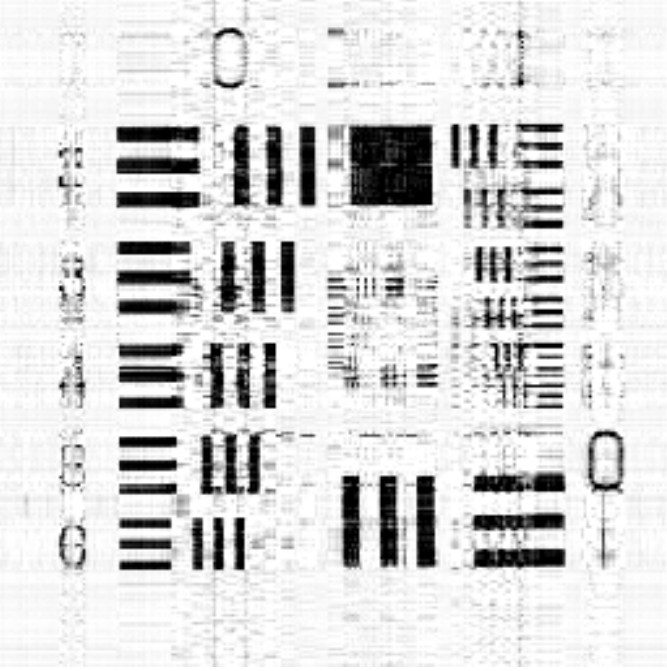}\vspace{0pt}
			\includegraphics[width=\linewidth]{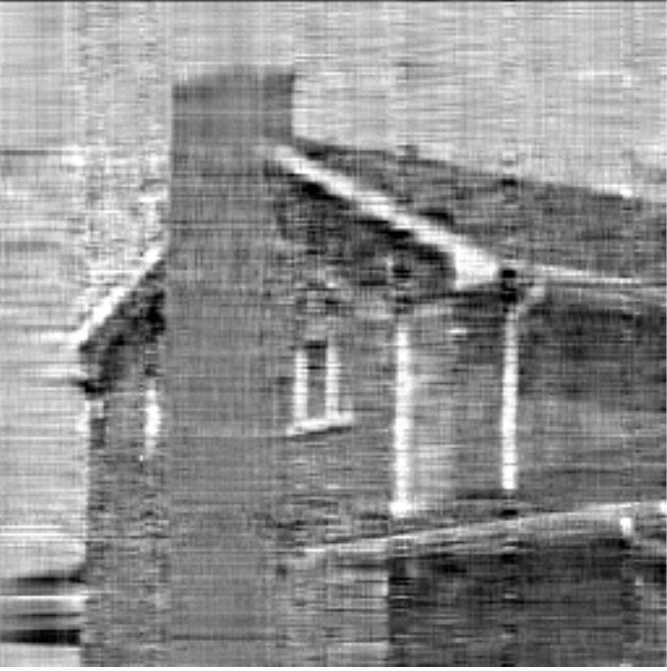}\vspace{0pt}
			\includegraphics[width=\linewidth]{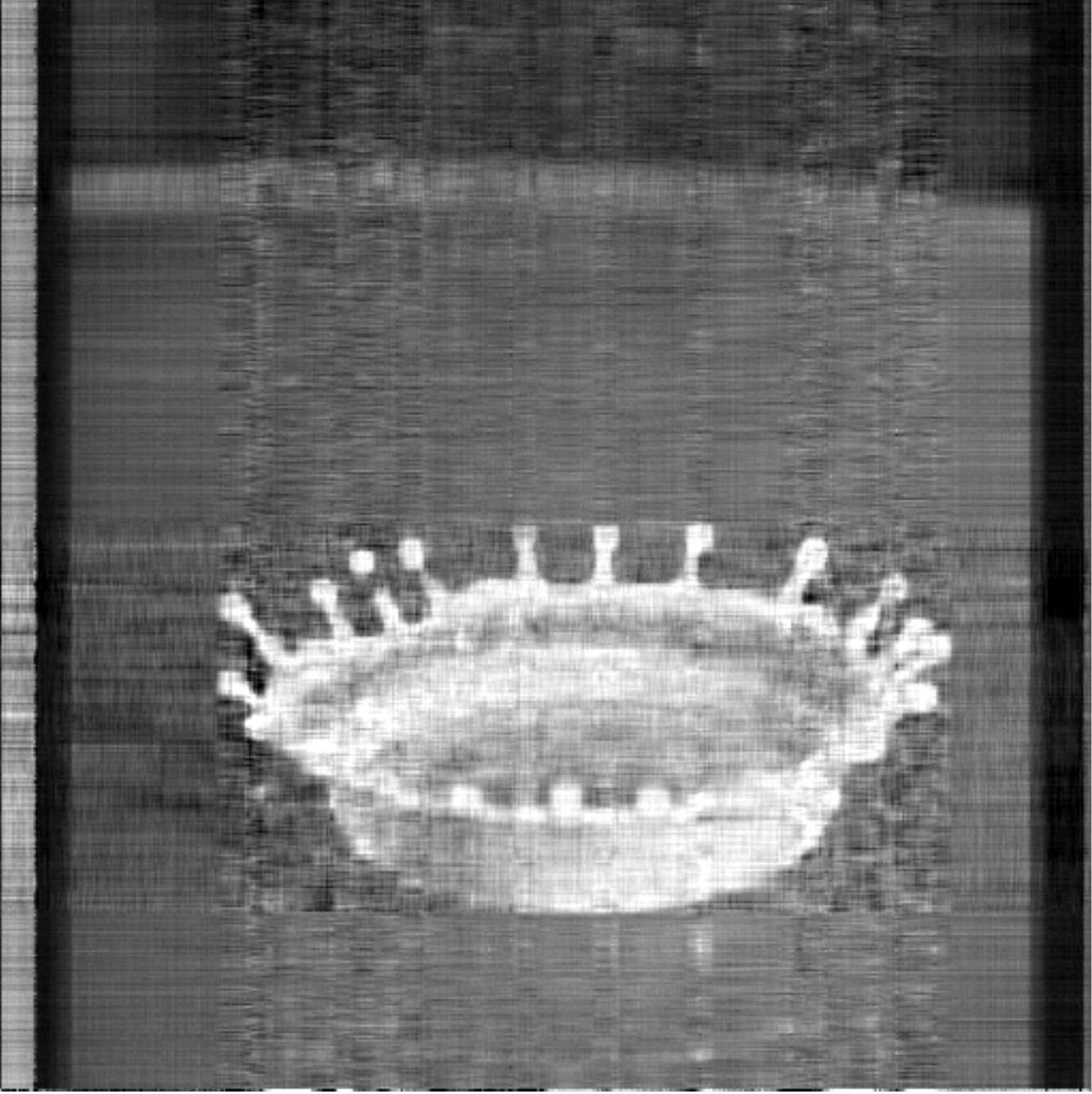}
			\caption{\tiny{FPCA}}
		\end{subfigure}
		\begin{subfigure}[b]{0.161\linewidth}
			\centering
			\includegraphics[width=\linewidth]{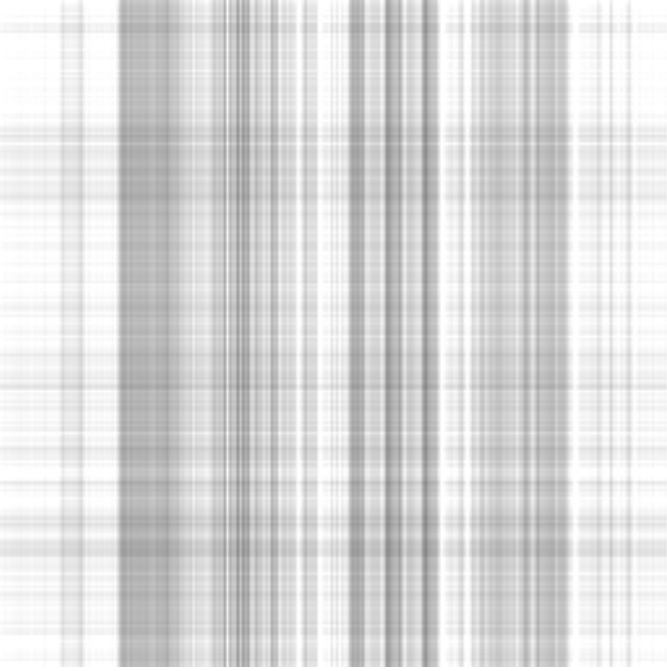}\vspace{0pt}
			\includegraphics[width=\linewidth]{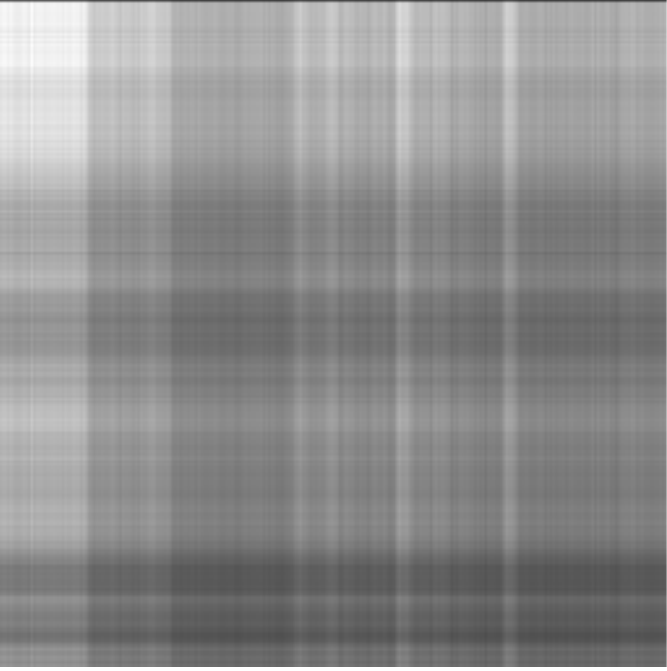}\vspace{0pt}
			\includegraphics[width=\linewidth]{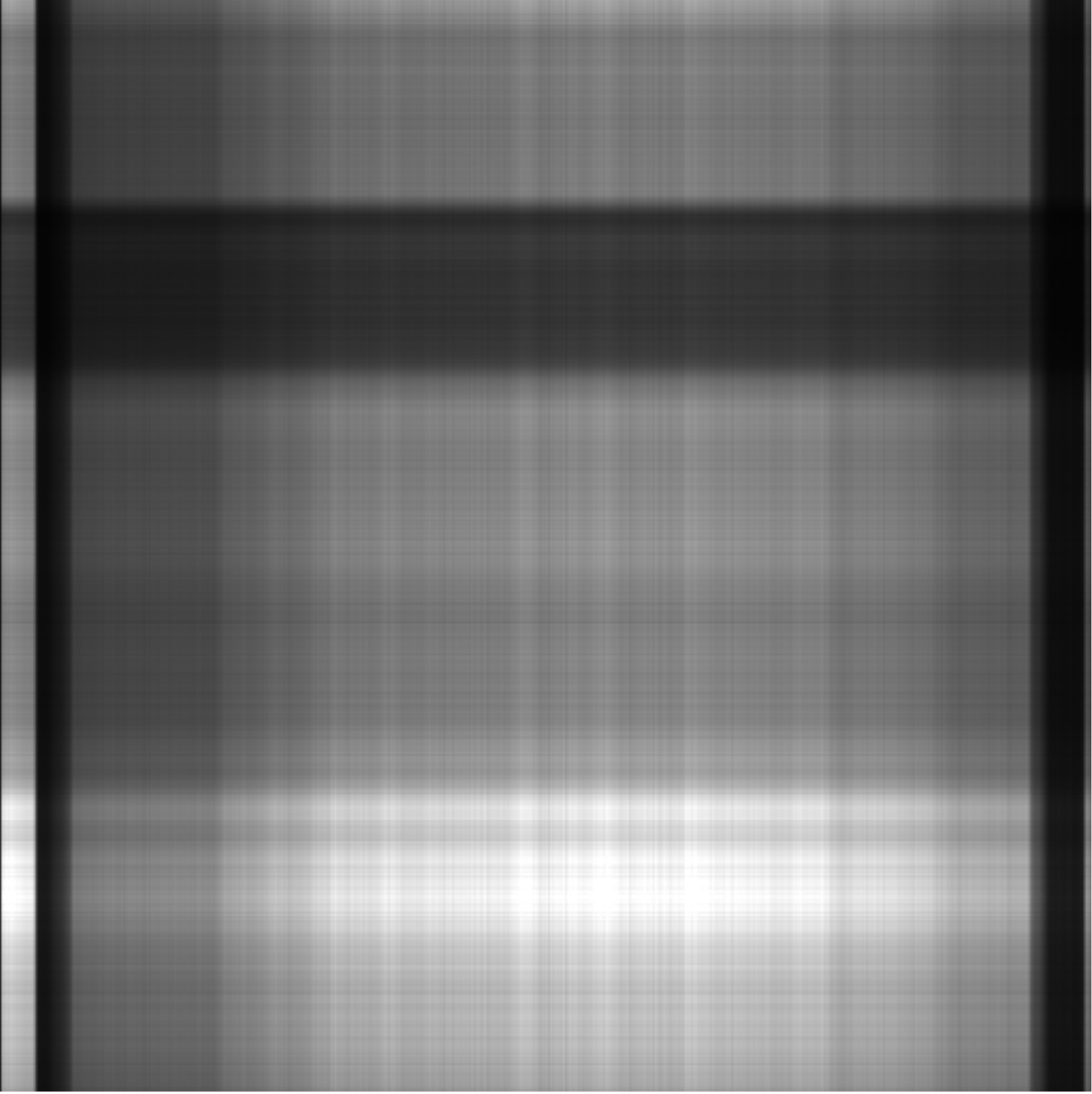}
			\caption{\tiny{SVT}}
		\end{subfigure}
		
	\end{subfigure}
	\vfill
	\caption{Image inpainting sample of image under a mixture of Gaussian noise.}
	\label{fig:images-mix-gau}
\end{figure}

\begin{table*}
	\caption{Image inpainting performance comparison under a mixture of Gaussian noise: PSNR and running times.}\label{tab:image-gmm}
	\centering
	\begin{tabular}{|c|cc|cc|cc|cc|cc}
		\hline
		{\multirow{2}{*}{\diagbox{Method}{Image}}}&\multicolumn{2}{c|}{Chart
		} &\multicolumn{2}{c|}{House}&\multicolumn{2}{c|}{Splash}
		\\ \cline{2-7}
		&  PSNR      & time &  PSNR      & time
		&PSNR         & time
		\\   \hline
		SPG &     26.21  &      7.33  &     29.19  &      9.34  &     31.92  &     39.02            \\		
		VBMFL1 &     21.43  &     65.67  &     28.08  &     53.84  &     30.88  &    183.73            \\
		
		FPCA &     16.96  &      7.45  &     20.35  &      4.75  &     24.05  &     24.64              \\
		
		SVT &     10.97  &     14.33  &     17.20  &      8.22  &     13.79  &     22.13             \\
		\hline
	\end{tabular}	
\end{table*}

In Figure \ref{fig:images-gau} and Table \ref{tab:image}, we consider the case where entries are missing at random by sampling ratio $ SR = 0.7 $, Gaussian noise with variance $0.0001$ is added to the observed pixels. It can be seen from Figure \ref{fig:images-gau} and Table \ref{tab:image} that the recovery effect of the algorithm based on $ l_1 $ norm is similar to that of the algorithm based on $ l_2 $ norm in Gaussian noise, and the running time is relatively long. However, SPG algorithm still has the highest PSNR value and shorter running time. Therefore, whether Gaussian noise or non-Gaussian noise, SPG algorithm performs best for image restoration.

\begin{figure}[htbp]
	\centering
	\begin{subfigure}[b]{1\linewidth}
		\begin{subfigure}[b]{0.161\linewidth}
			\centering
			\includegraphics[width=\linewidth]{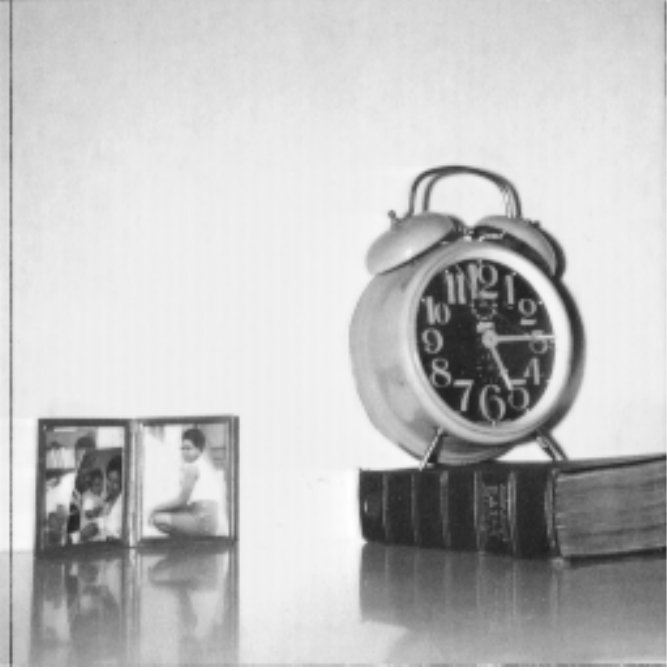}\vspace{0pt}
			\includegraphics[width=\linewidth]{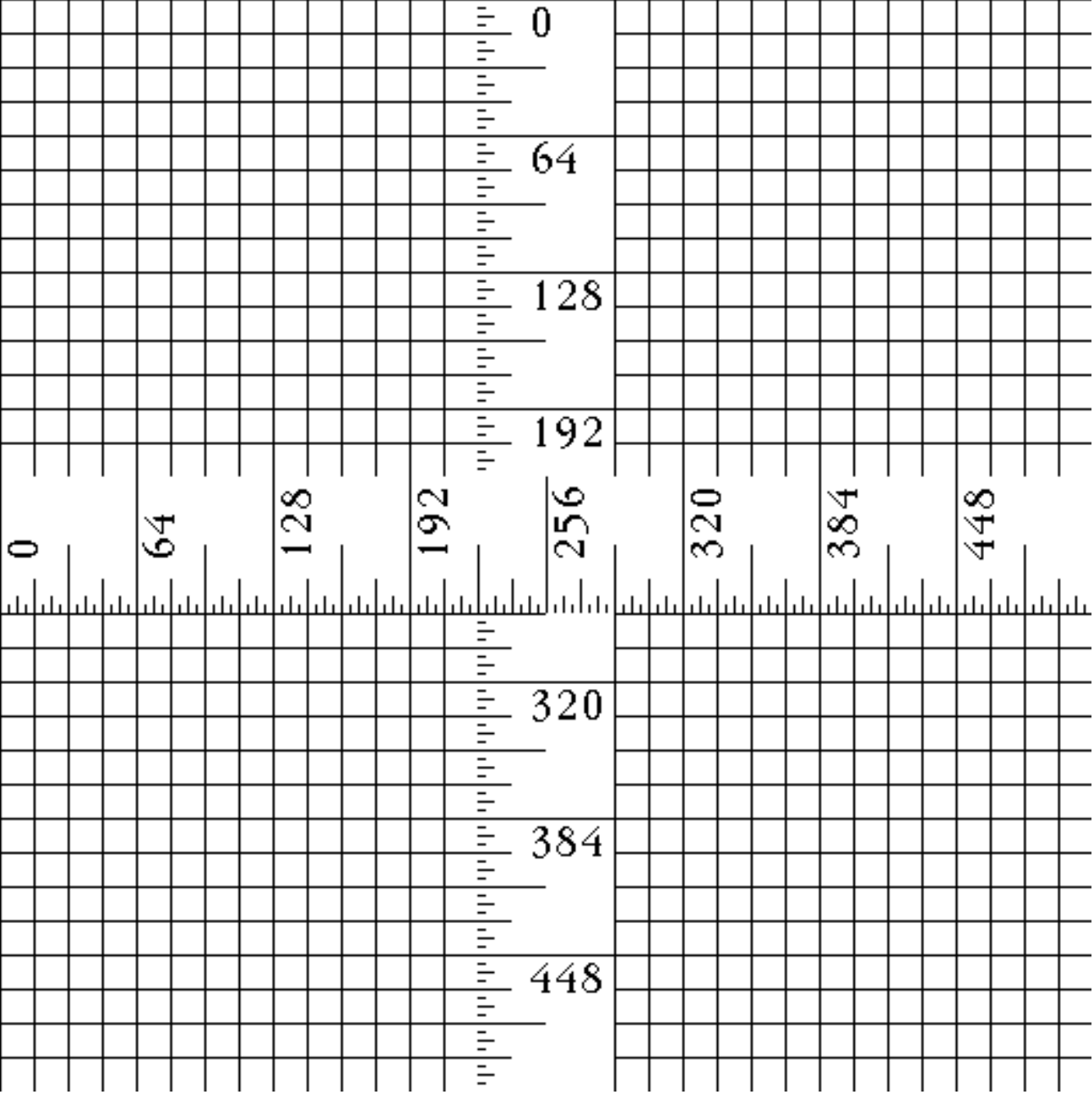}\vspace{0pt}
			\includegraphics[width=\linewidth]{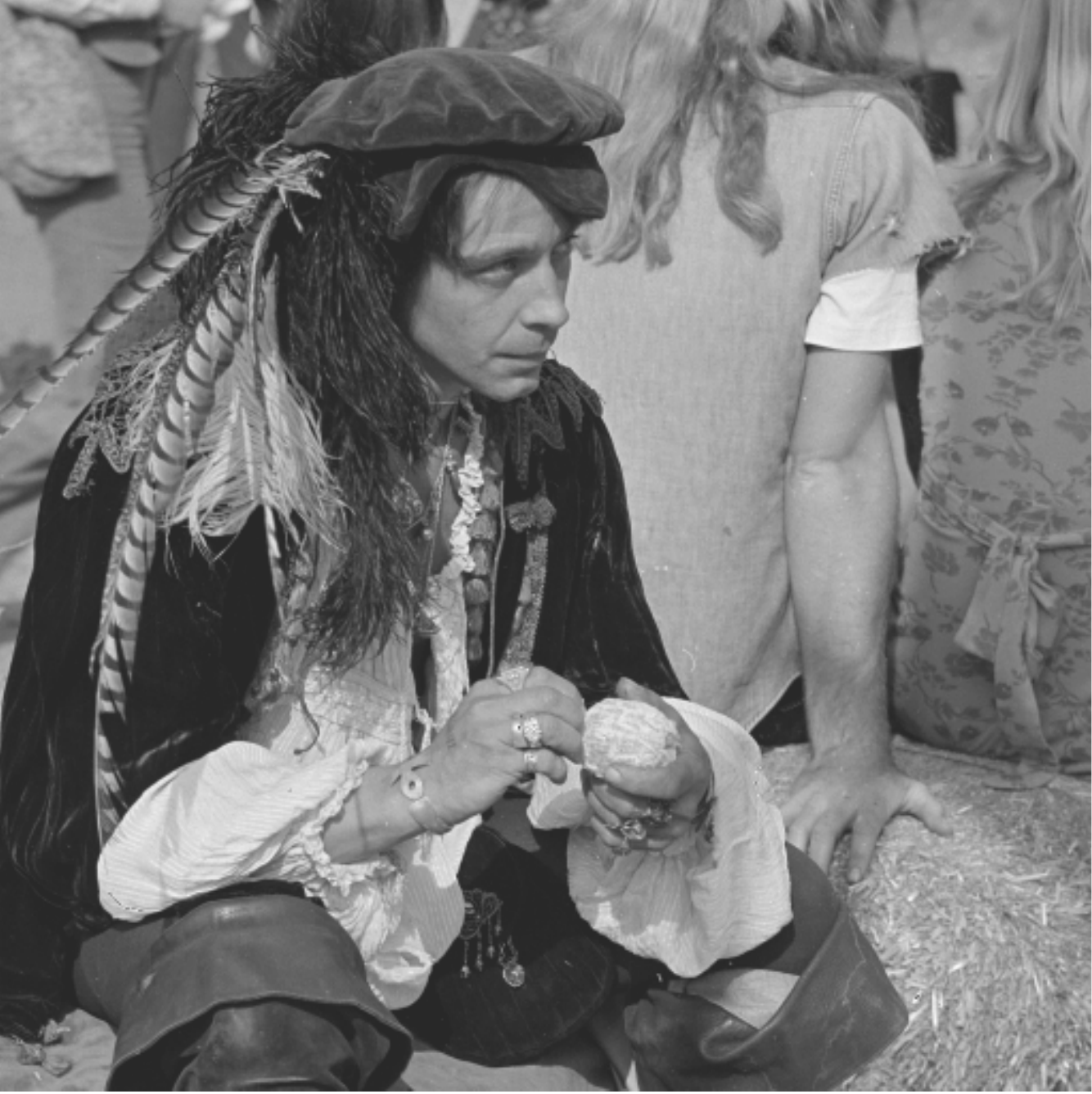}
			\caption{\tiny{Original}}
		\end{subfigure}
		\begin{subfigure}[b]{0.161\linewidth}
			\centering
			\includegraphics[width=\linewidth]{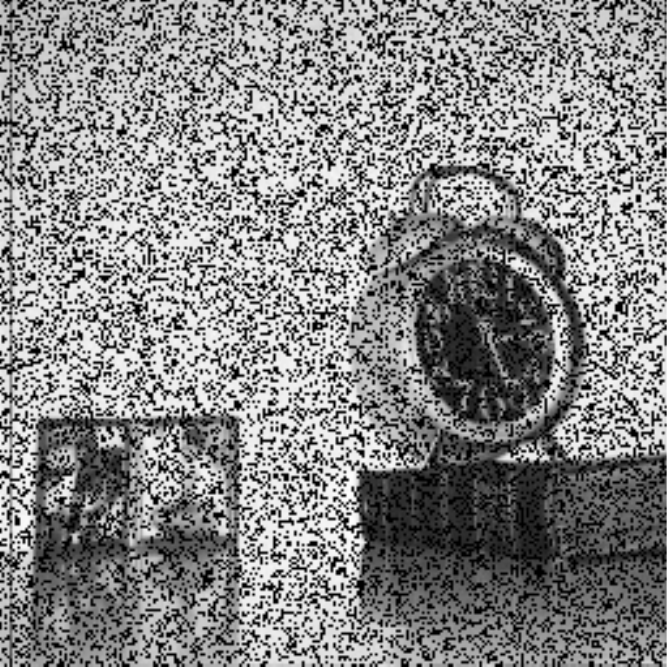}\vspace{0pt}
			\includegraphics[width=\linewidth]{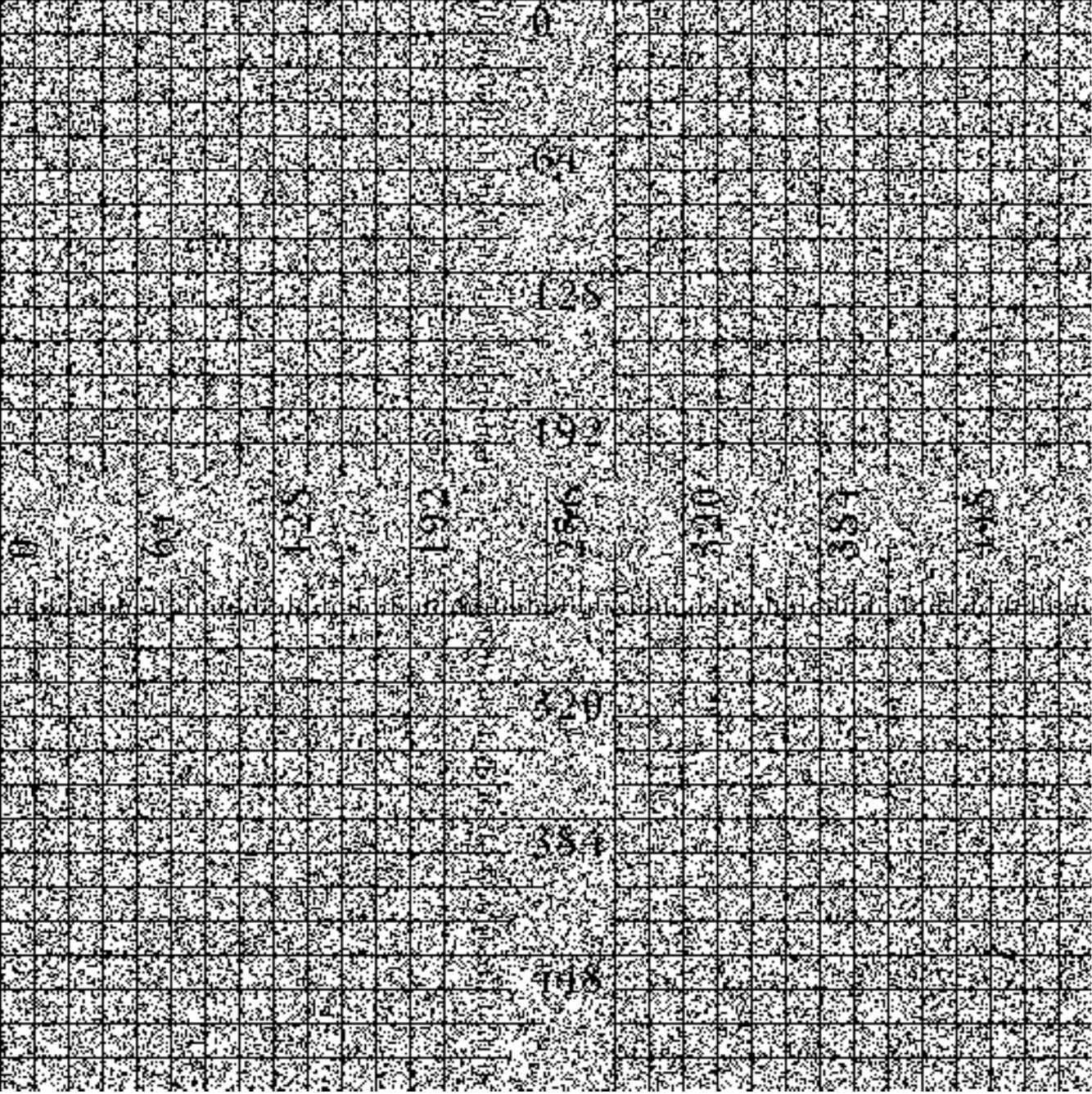}\vspace{0pt}
			\includegraphics[width=\linewidth]{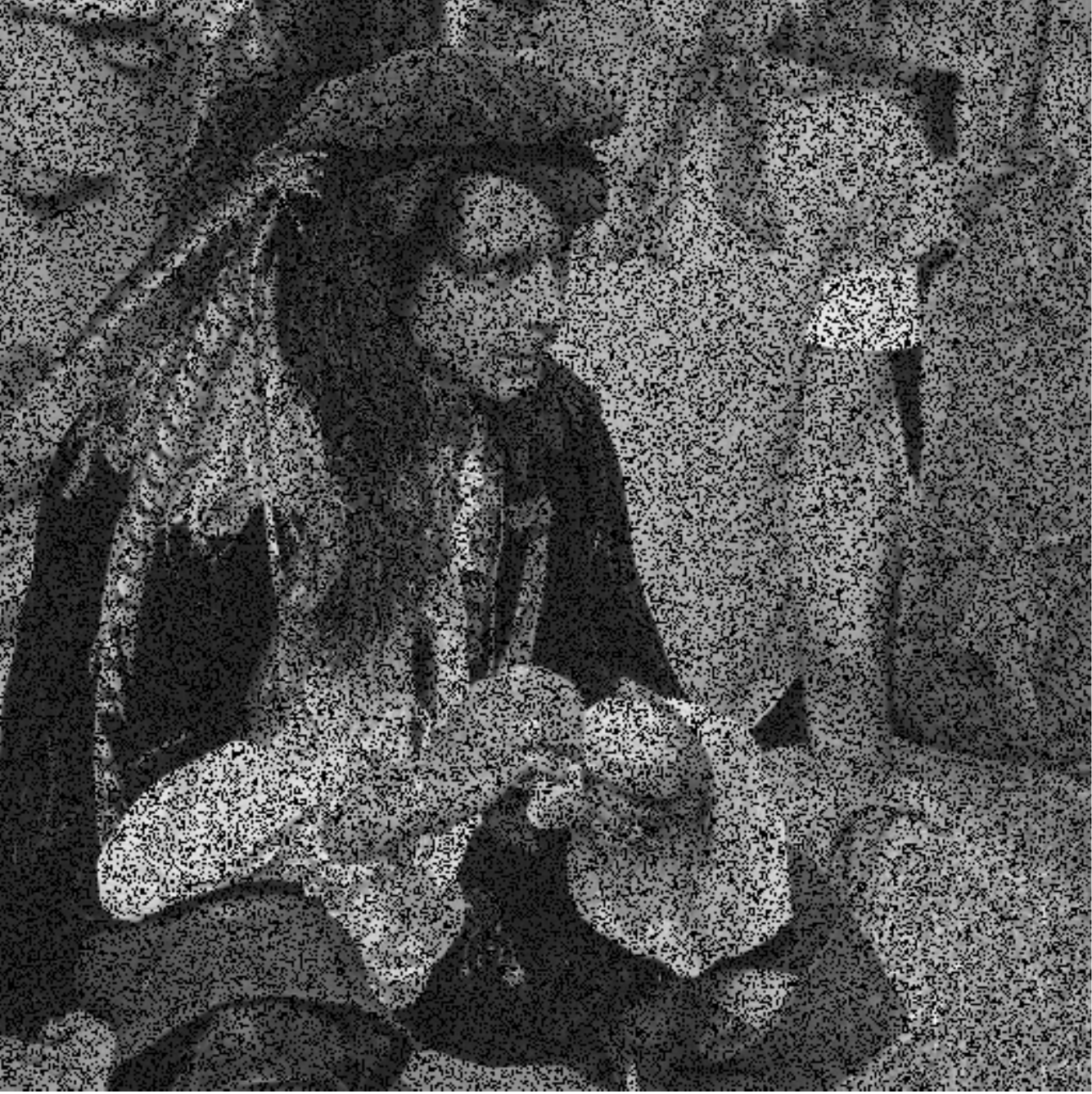}
			\caption{\tiny{Observation}}
		\end{subfigure}
		\begin{subfigure}[b]{0.161\linewidth}
			\centering
			\includegraphics[width=\linewidth]{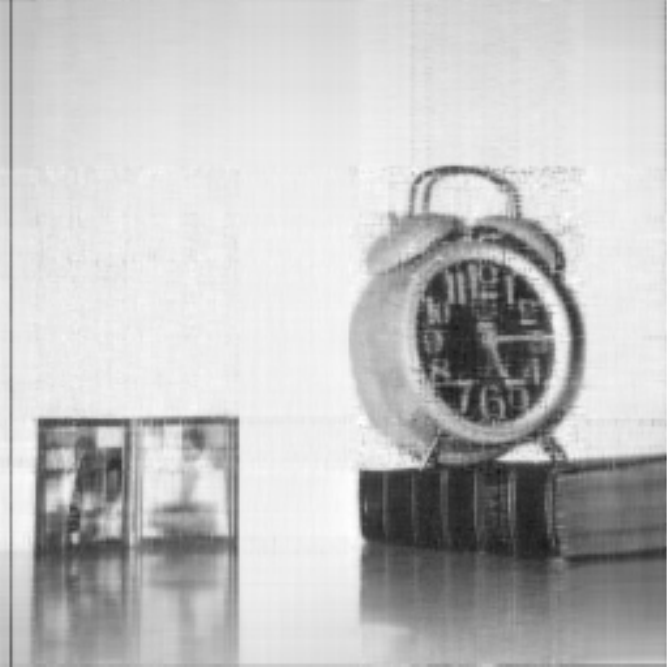}\vspace{0pt}
			\includegraphics[width=\linewidth]{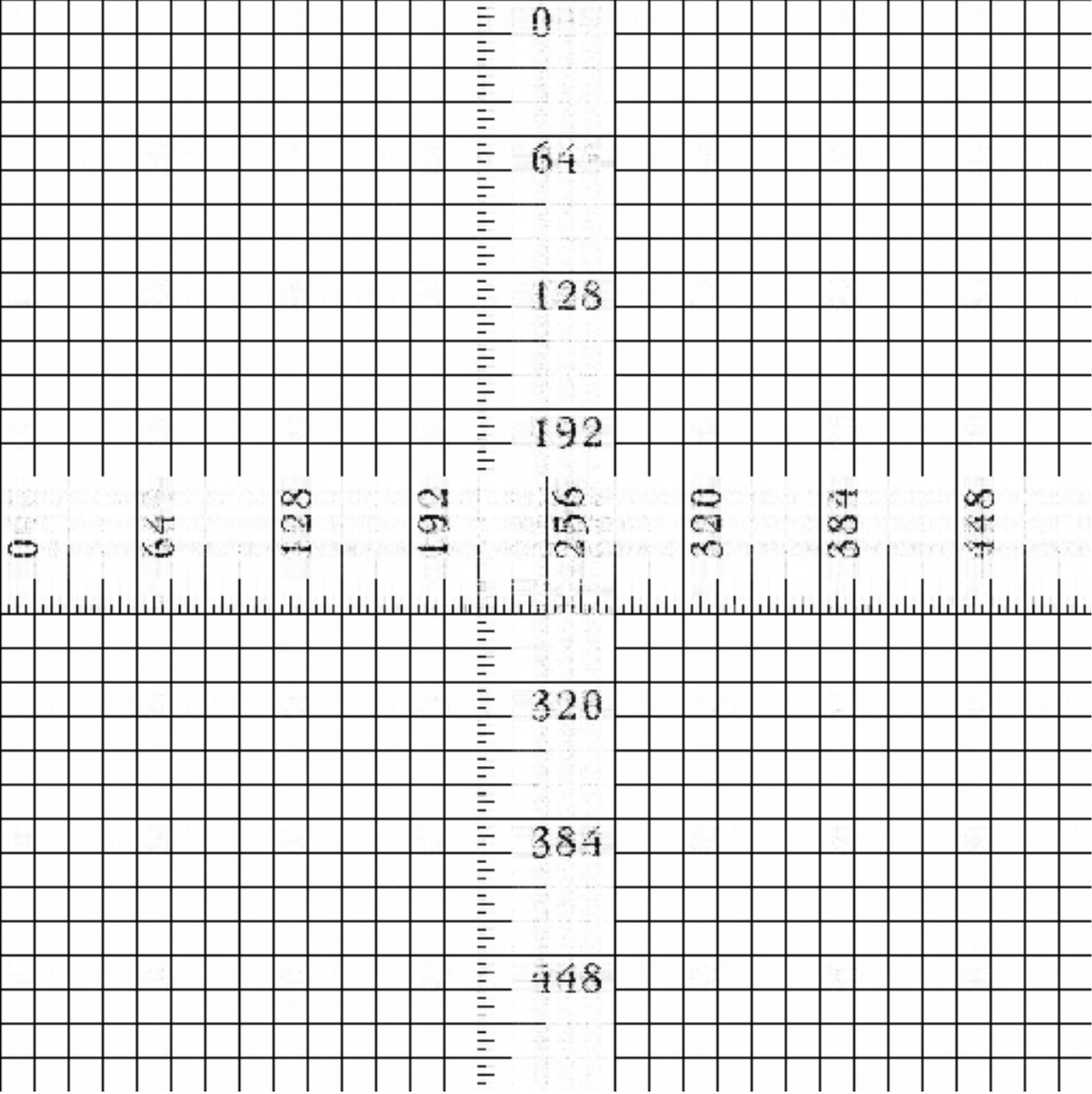}\vspace{0pt}
			\includegraphics[width=\linewidth]{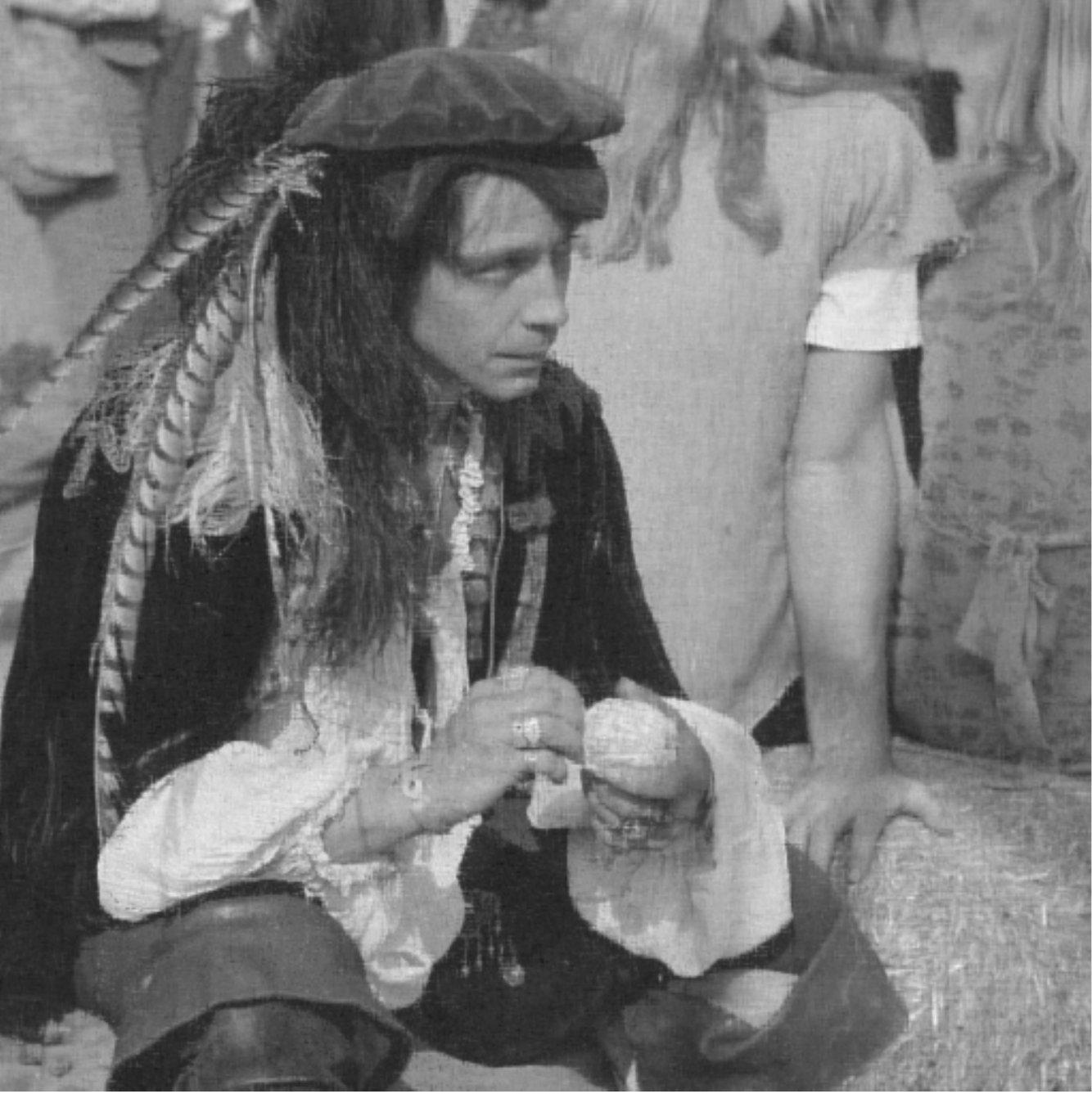}
			\caption{\tiny{SPG}}
		\end{subfigure}
		\begin{subfigure}[b]{0.161\linewidth}
			\centering
			\includegraphics[width=\linewidth]{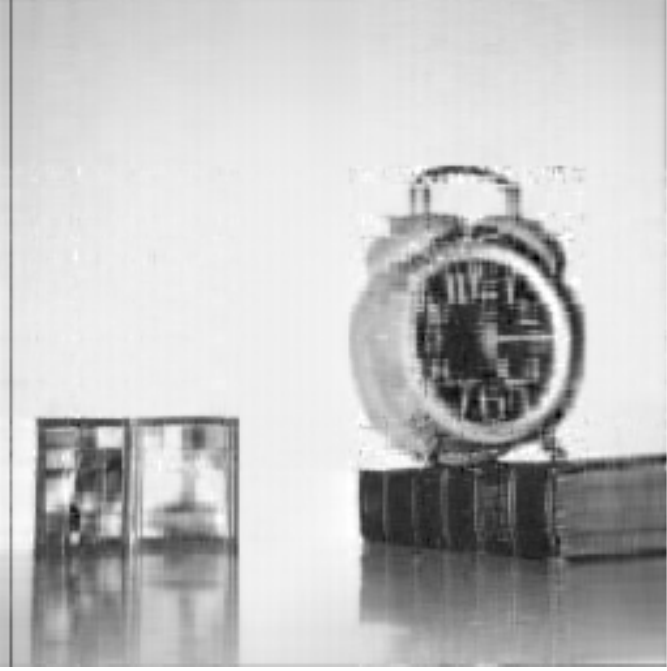}\vspace{0pt}
			\includegraphics[width=\linewidth]{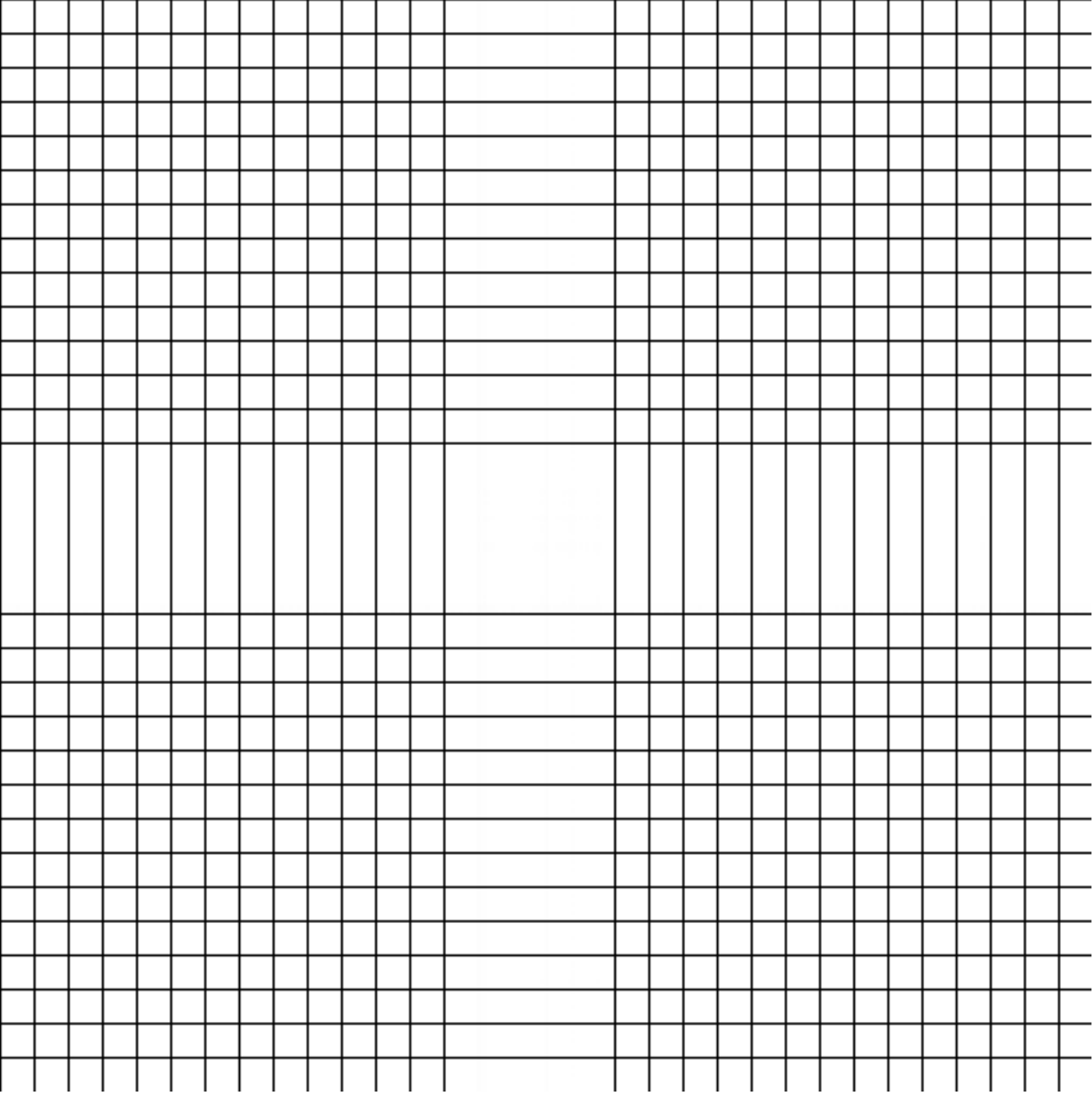}\vspace{0pt}
			\includegraphics[width=\linewidth]{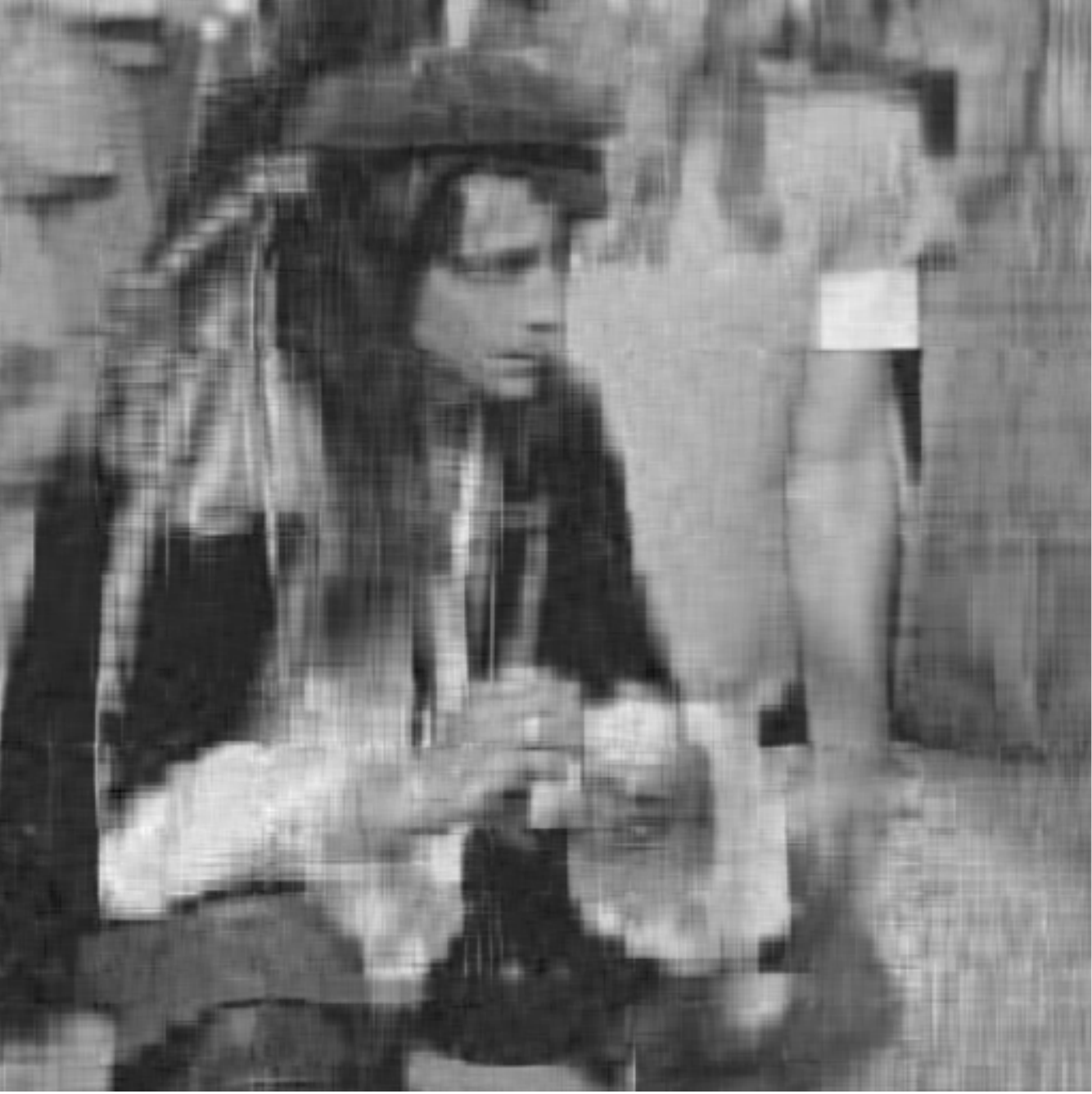}
			\caption{\tiny{VBMFL1}}
		\end{subfigure}
		\begin{subfigure}[b]{0.161\linewidth}
			\centering
			\includegraphics[width=\linewidth]{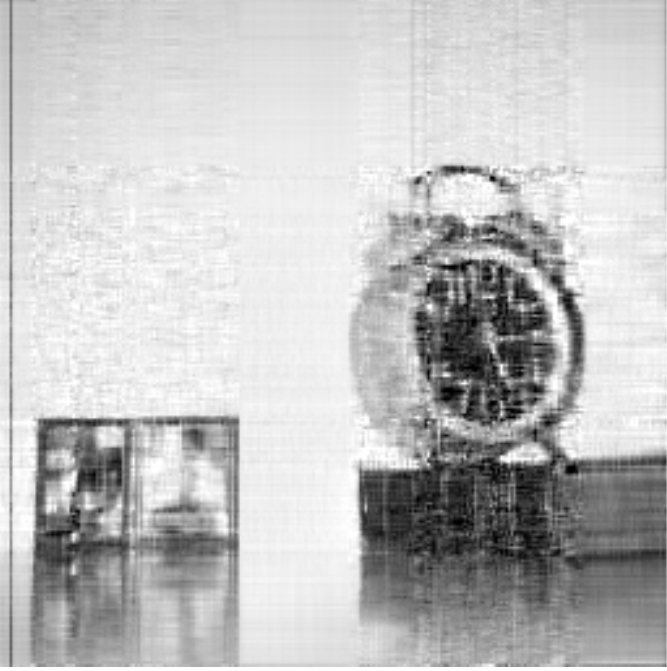}\vspace{0pt}
			\includegraphics[width=\linewidth]{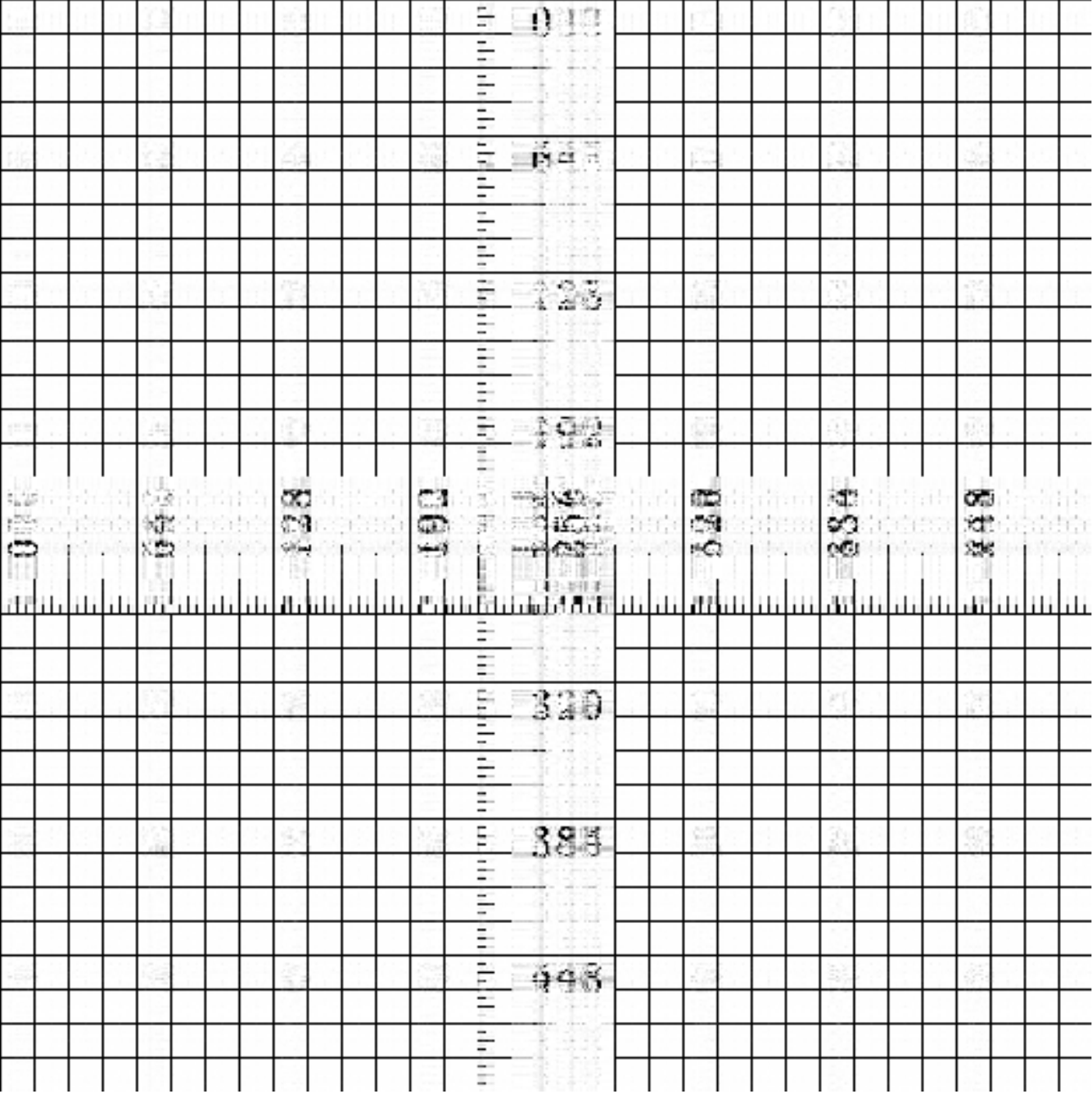}\vspace{0pt}
			\includegraphics[width=\linewidth]{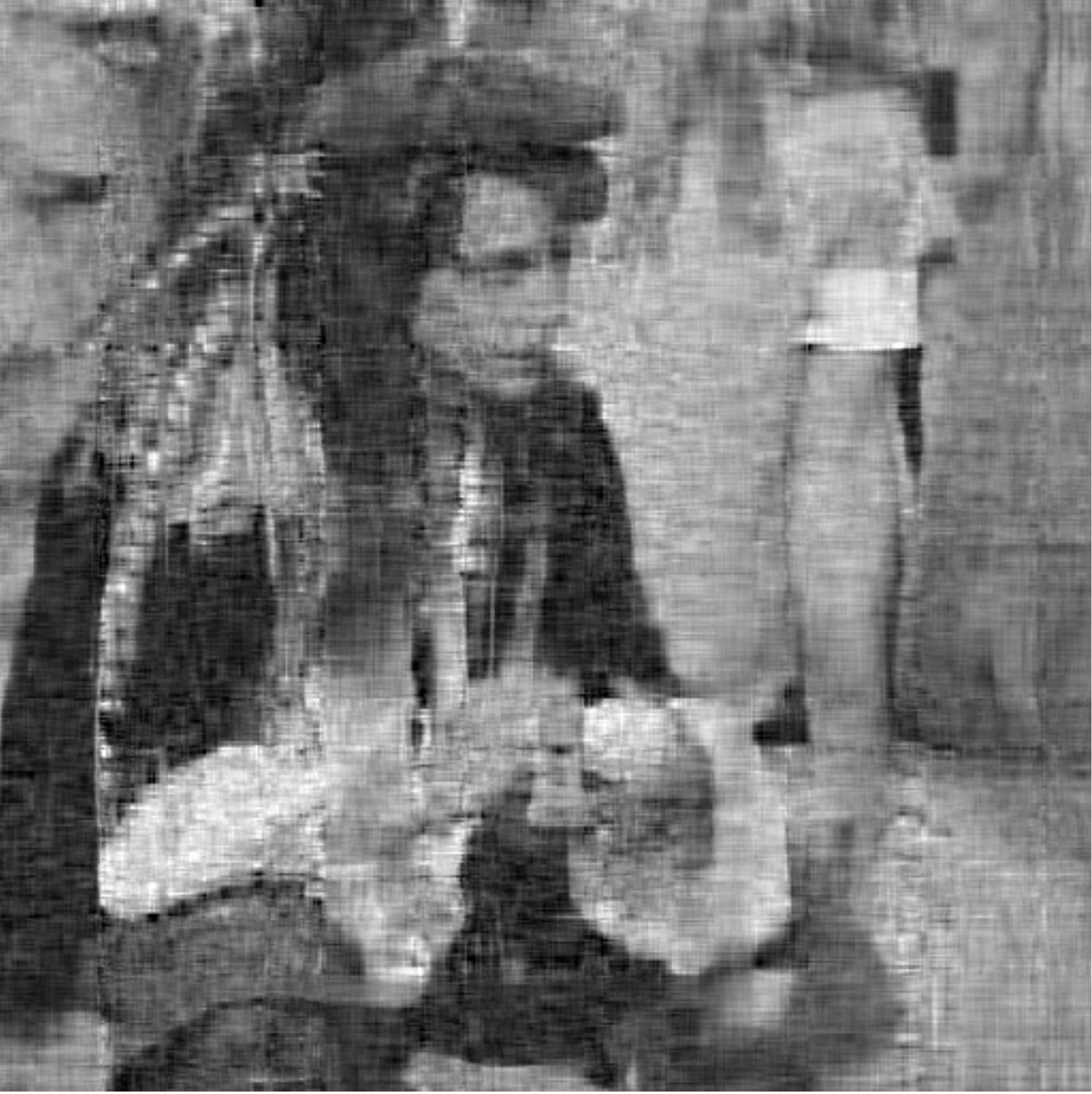}
			\caption{\tiny{FPCA}}
		\end{subfigure}
		\begin{subfigure}[b]{0.161\linewidth}
			\centering
			\includegraphics[width=\linewidth]{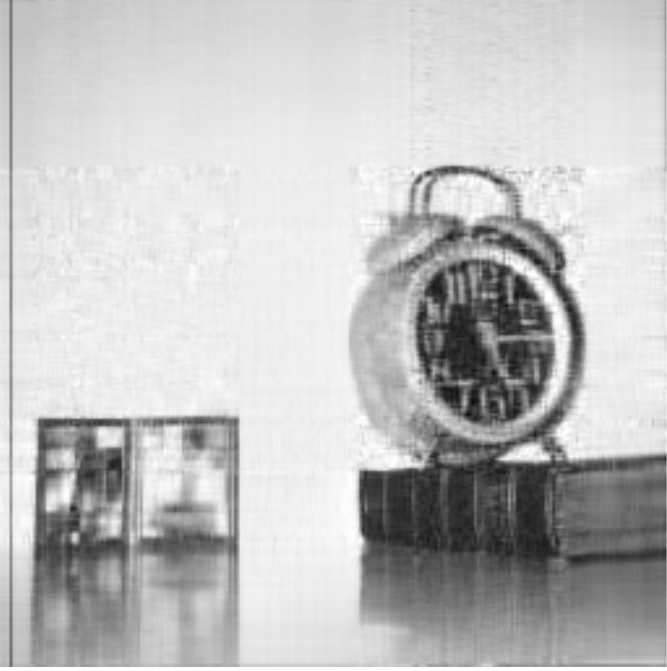}\vspace{0pt}
			\includegraphics[width=\linewidth]{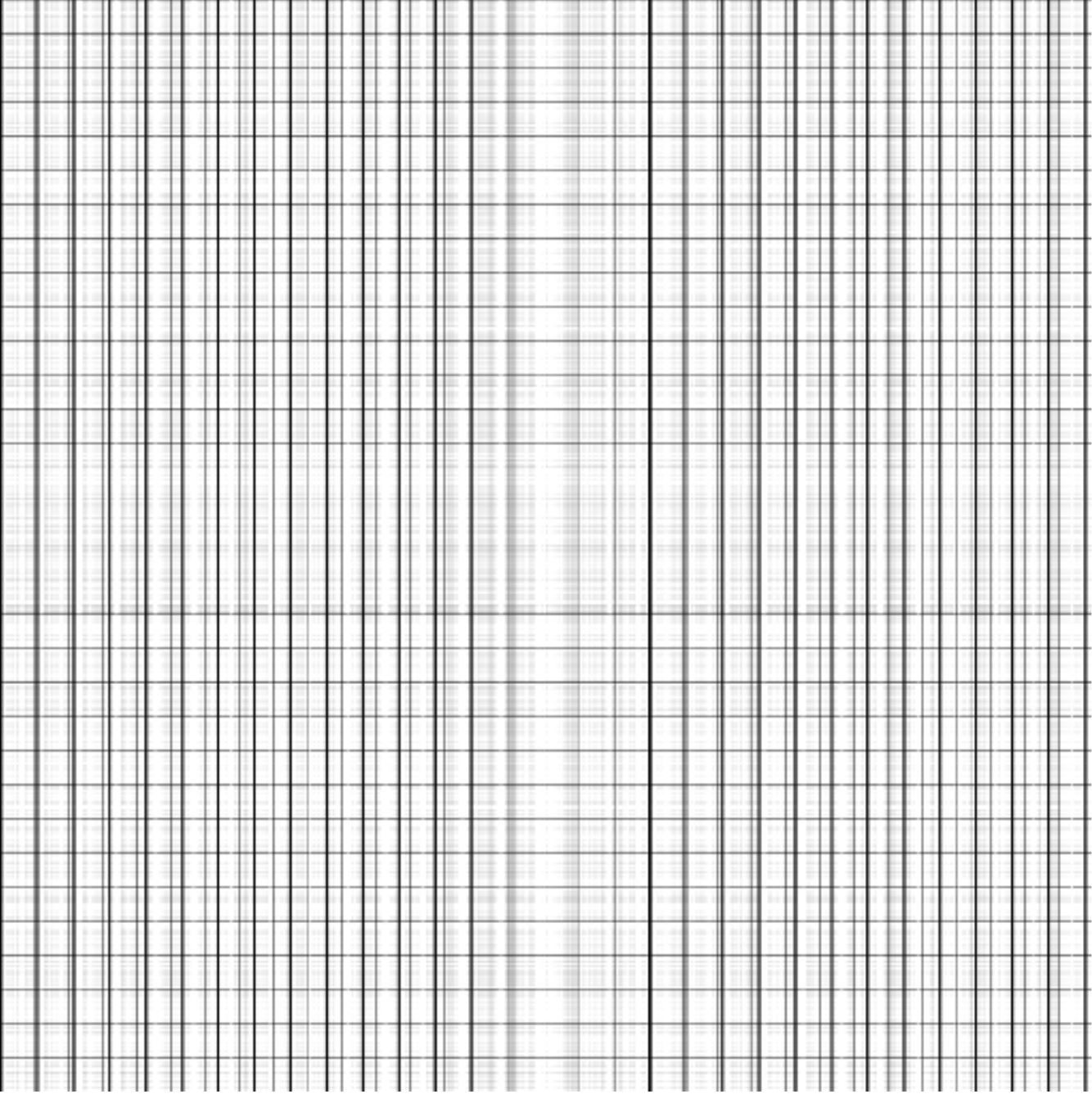}\vspace{0pt}
			\includegraphics[width=\linewidth]{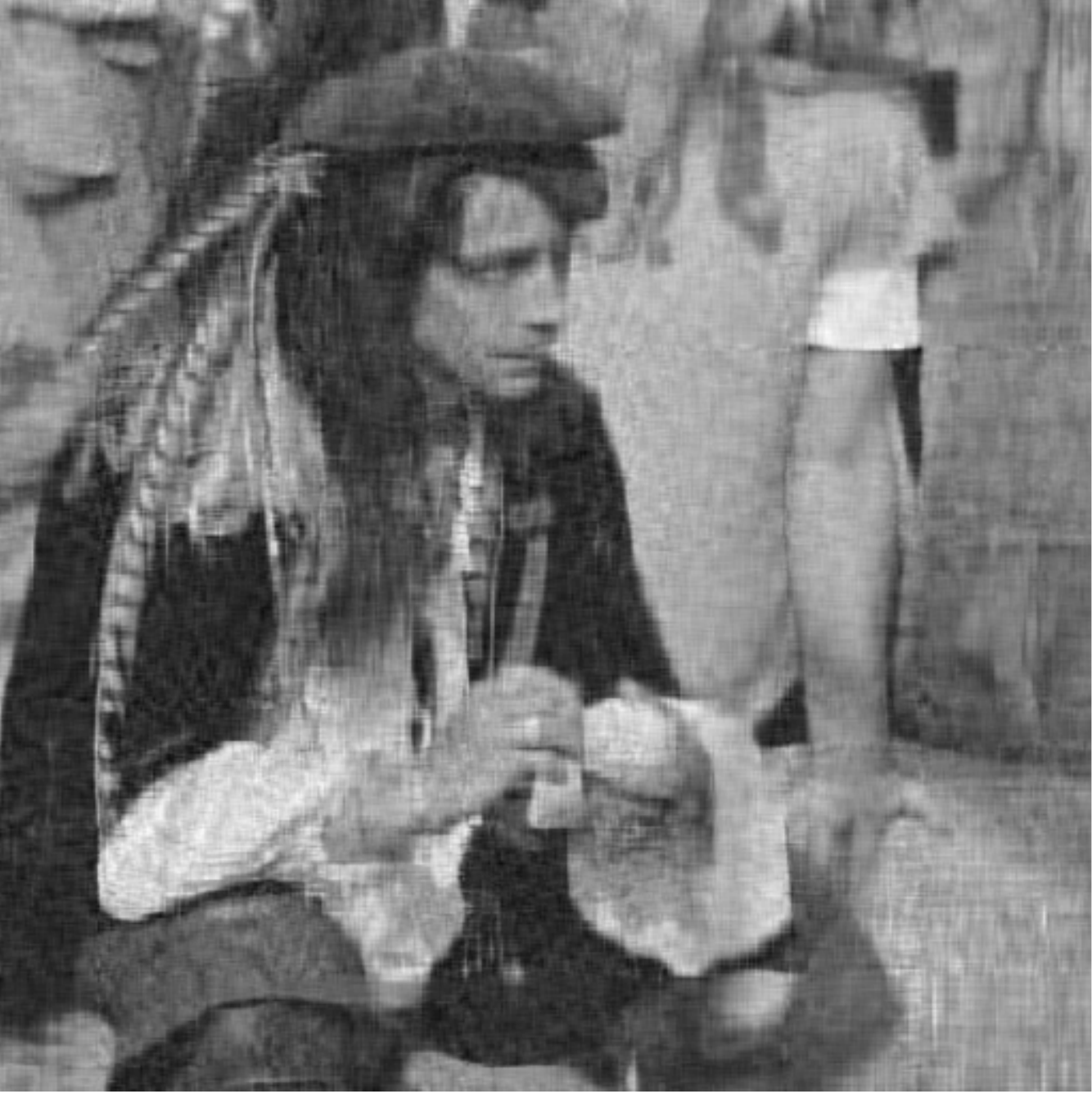}
			\caption{\tiny{SVT}}
		\end{subfigure}
		
	\end{subfigure}
	\vfill
	\caption{Image inpainting sample of image under Gaussian noise.}
	\label{fig:images-gau}
\end{figure}

\begin{table*}
	\caption{Image inpainting performance comparison under Gaussian noise: PSNR and running times.}\label{tab:image}
	\centering
	\begin{tabular}{|c|cc|cc|cc|cc|cc}
		\hline
		{\multirow{2}{*}{\diagbox{Method}{Image}}}&\multicolumn{2}{c|}{Clock
		} &\multicolumn{2}{c|}{Ruler}&\multicolumn{2}{c|}{Man}
		\\ \cline{2-7}
		&  PSNR      & time &  PSNR      & time
		&PSNR         & time
		\\   \hline
		SPG &     29.65  &      3.62  &     28.65  &     12.37  &     30.11  &     21.17  \\		
		VBMFL1 &     28.23  &     25.48  &     17.88  &     15.73  &     23.82  &    114.15  \\
		
		FPCA &     24.13  &      4.82  &     21.75  &     23.02  &     22.68  &     22.58  \\
		
		SVT &     28.87  &     27.64  &     10.37  &     40.73  &     25.86  &     81.60  \\
		\hline
	\end{tabular}	
\end{table*}

\subsection{MRI Volume Dataset}
The resolution of the MRI volume dataset\footnote{\url{http://graphics.stanford.edu/data/voldata/}} is of size $ 217 \times 181 $ with 181 slices and we selected the 38th slice and the 88th slice for the experiment. We consider the case where entries are missing at random by sampling ratio $ SR = 0.9 $. The GMM noise are set at $\sigma_{A}^{2}=0.0001, \sigma_{B}^{2}=0.1, c=0.01$.

From Figure \ref{fig:images-MRI} and Histogram \ref{his:MRI}, we can see that the effect of FPCA and SVT to restore images is very poor. The effect of VBMFL1 algorithm to restore images is good, but the running time is relatively long. SPG algorithm to restore the image effect and good running time is short. In summary, the SPG algorithm has the best recovery effect.

\begin{figure}[htbp]
	\centering
	\begin{subfigure}[b]{1\linewidth}
		\begin{subfigure}[b]{0.161\linewidth}
			\centering
			\includegraphics[width=\linewidth]{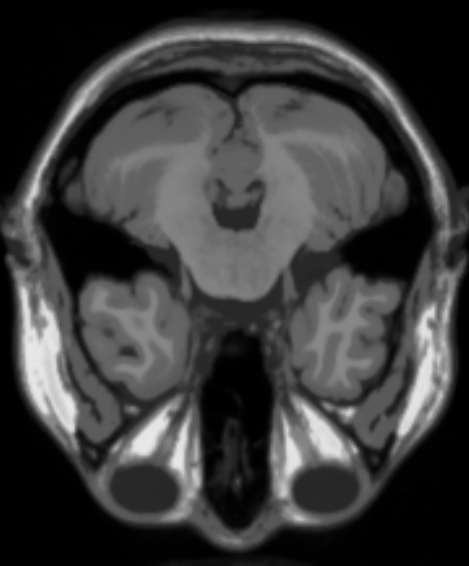}\vspace{0pt}
			\includegraphics[width=\linewidth]{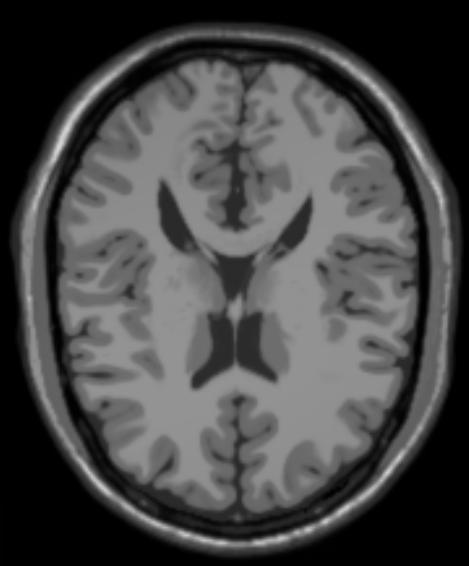}
			\caption{\tiny{Original}}
		\end{subfigure}
		\begin{subfigure}[b]{0.161\linewidth}
			\centering
			\includegraphics[width=\linewidth]{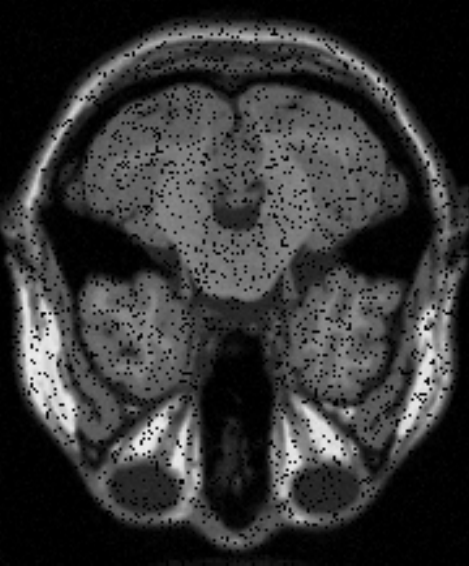}\vspace{0pt}
			\includegraphics[width=\linewidth]{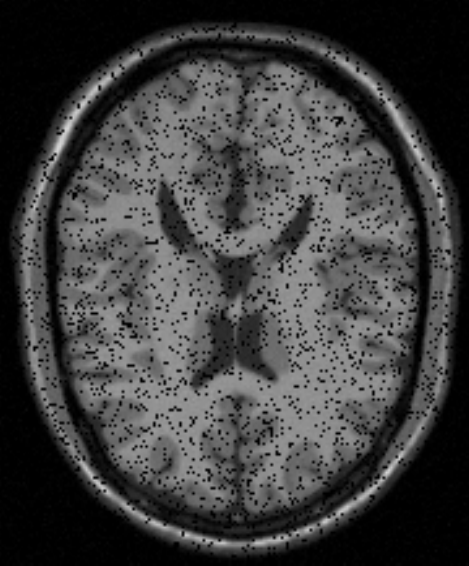}
			\caption{\tiny{Observation}}
		\end{subfigure}
		\begin{subfigure}[b]{0.161\linewidth}
			\centering
			\includegraphics[width=\linewidth]{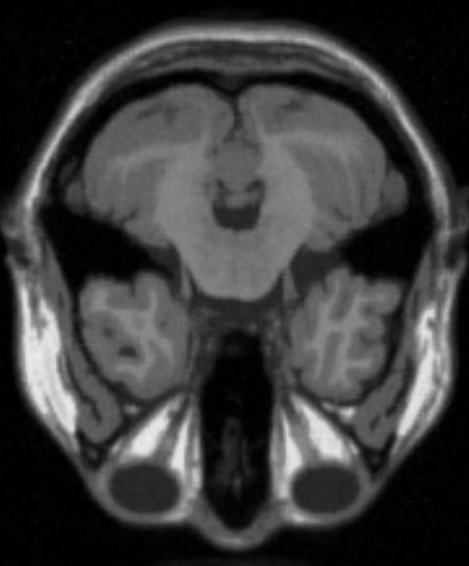}\vspace{0pt}
			\includegraphics[width=\linewidth]{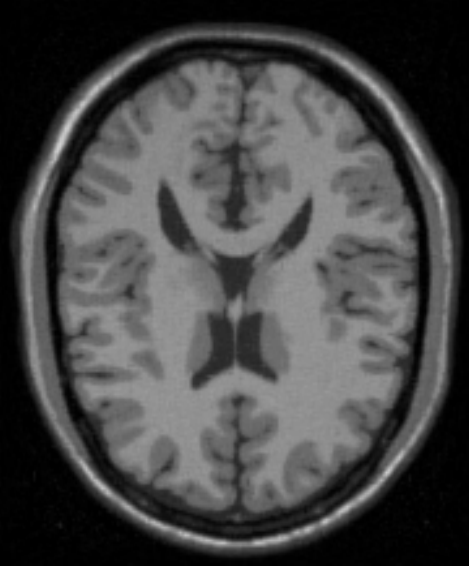}
			\caption{\tiny{SPG}}
		\end{subfigure}
		\begin{subfigure}[b]{0.161\linewidth}
			\centering
			\includegraphics[width=\linewidth]{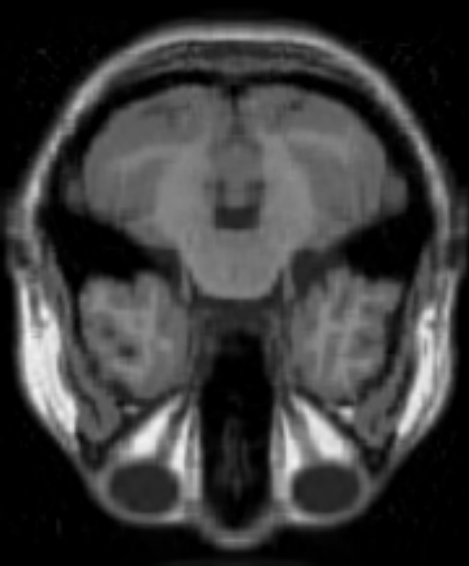}\vspace{0pt}
			\includegraphics[width=\linewidth]{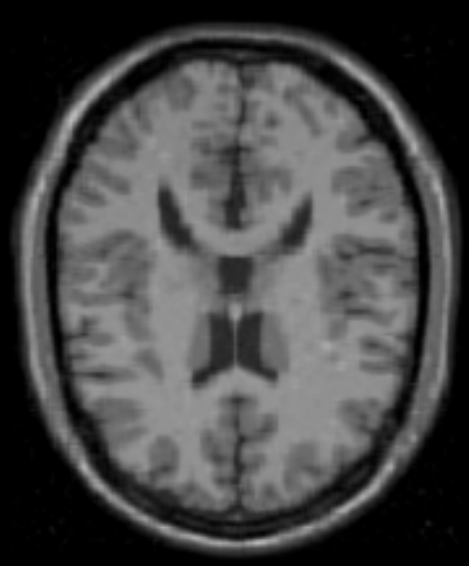}
			\caption{\tiny{VBMFL1}}
		\end{subfigure}
		\begin{subfigure}[b]{0.161\linewidth}
			\centering
			\includegraphics[width=\linewidth]{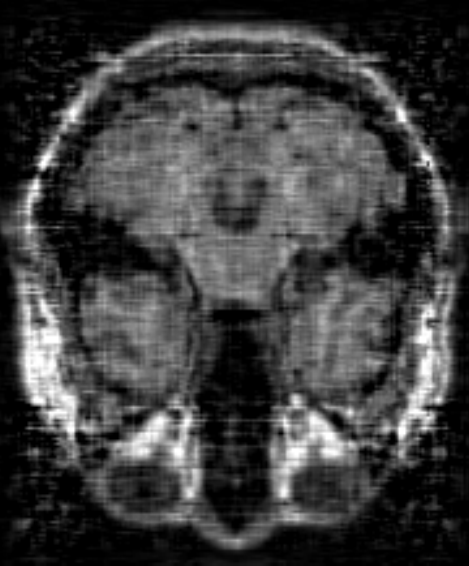}\vspace{0pt}
			\includegraphics[width=\linewidth]{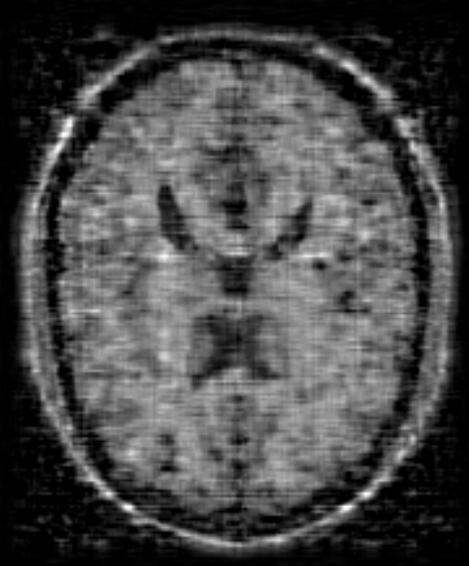}
			\caption{\tiny{FPCA}}
		\end{subfigure}
		\begin{subfigure}[b]{0.161\linewidth}
			\centering
			\includegraphics[width=\linewidth]{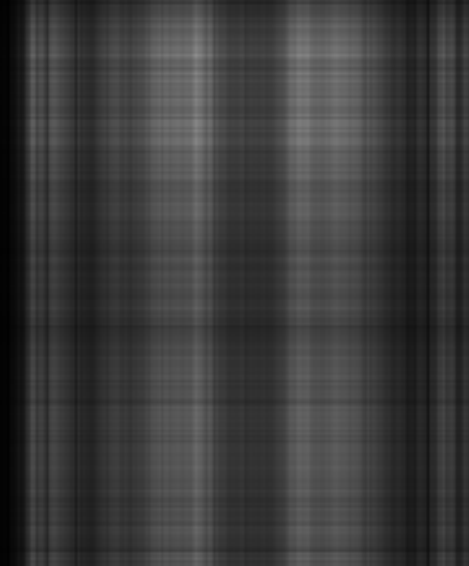}\vspace{0pt}
			\includegraphics[width=\linewidth]{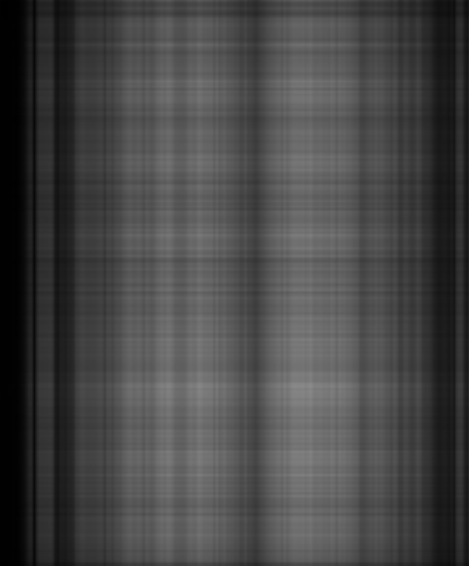}
			\caption{\tiny{SVT}}
		\end{subfigure}
		
	\end{subfigure}
	\vfill
	\caption{Completion results of the MRI Volume Dataset.}
	\label{fig:images-MRI}
\end{figure}

\begin{figure}[htbp]
	\centering
	\begin{subfigure}[b]{1\linewidth}
		\begin{subfigure}[b]{0.48\linewidth}
			\centering
			\includegraphics[width=\linewidth]{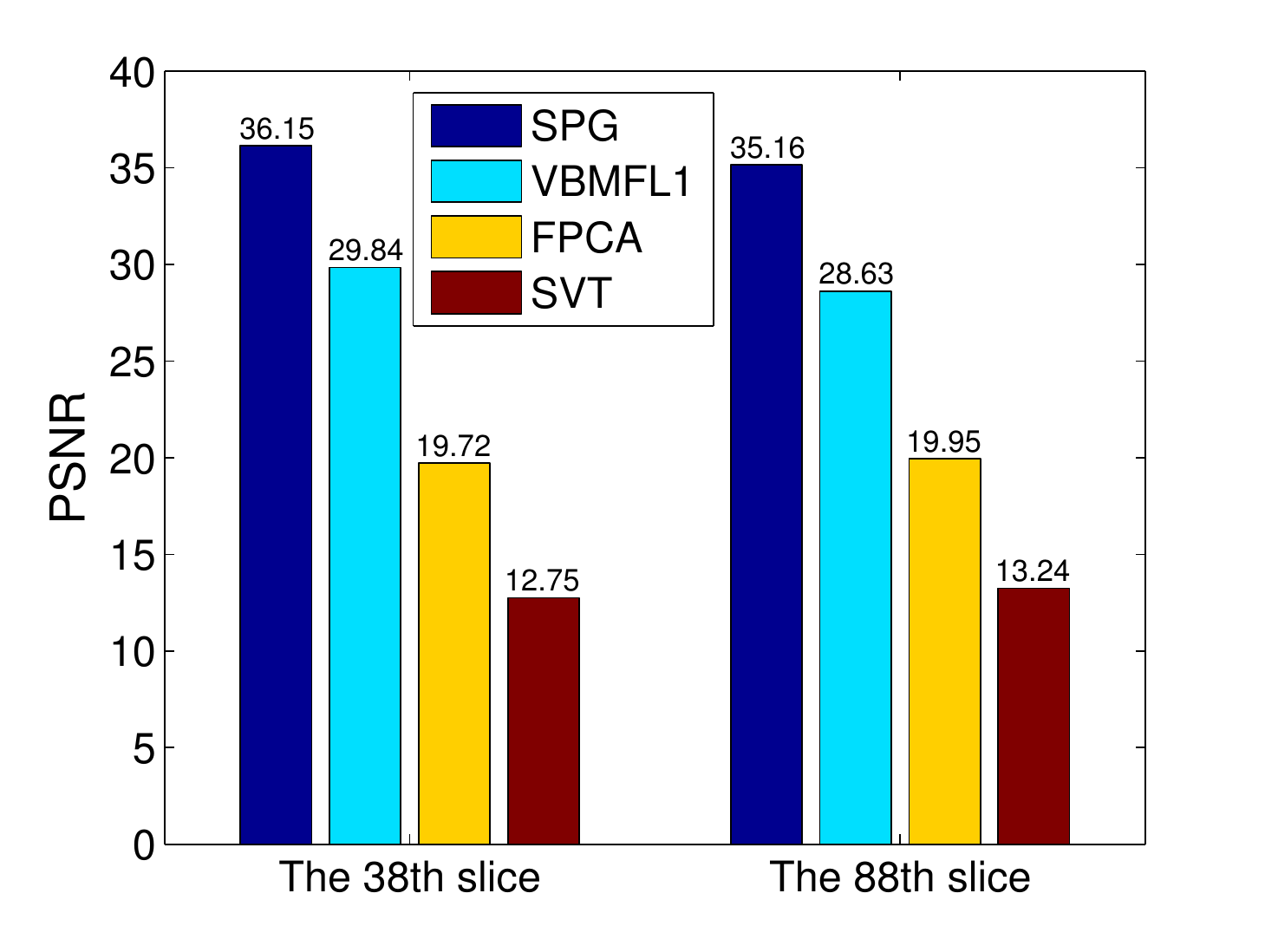}
		\end{subfigure}  	
		\begin{subfigure}[b]{0.48\linewidth}
			\centering
			\includegraphics[width=\linewidth]{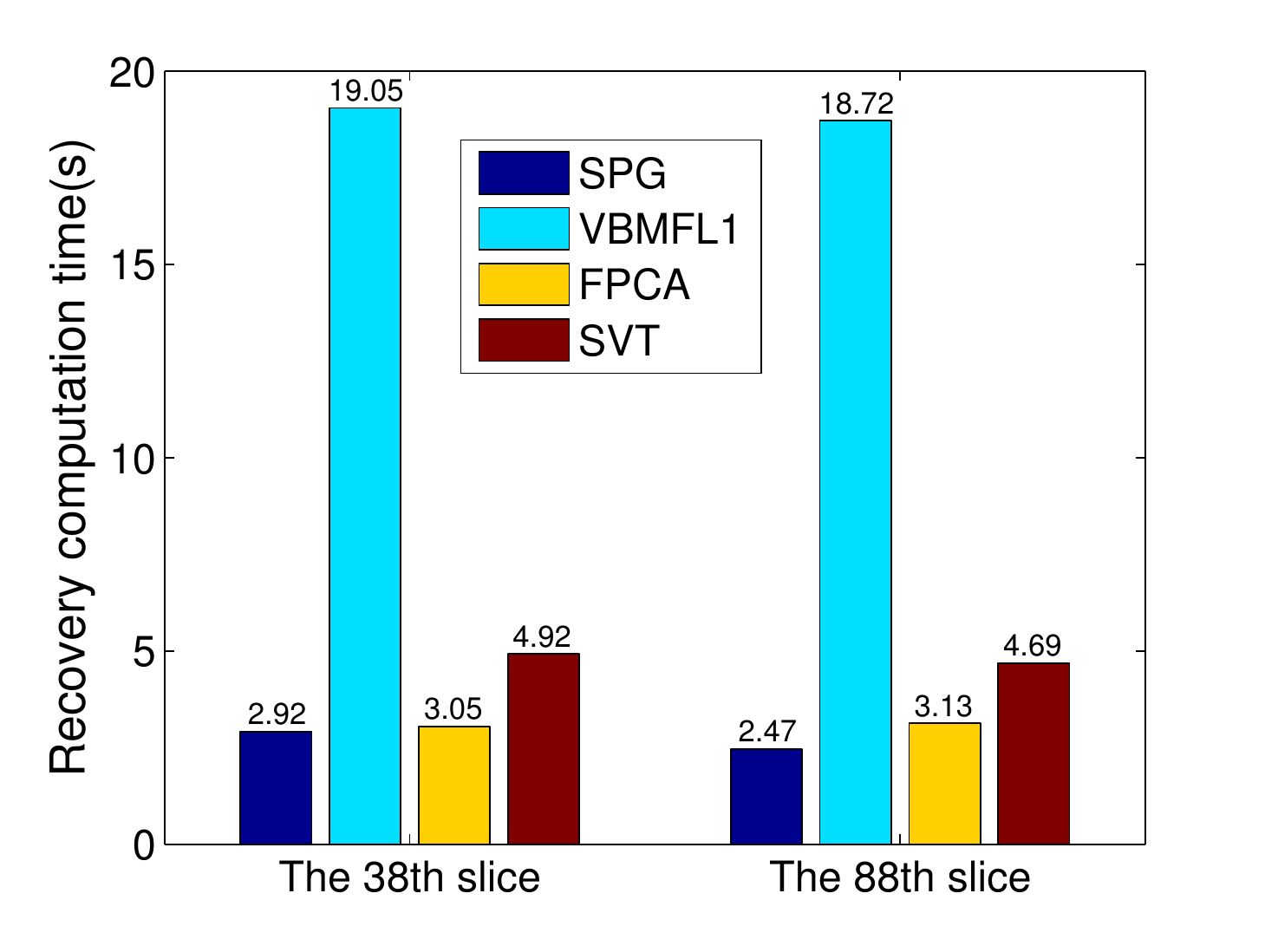}
		\end{subfigure}
	\end{subfigure}
	\vfill
	\caption{Histogram of representation results for the MRI Volume Dataset.}
	\label{his:MRI}
\end{figure}

\bigskip
\noindent
{\bf Acknowledgement}
Xinzhen Zhang was partly supported by
the National Natural Science Foundation of China
(Grant No. 11871369).
Quan Yu was partly supported by Tianjin Research Innovation Project for Postgraduate Students
(Grant No. 2020YJSS140).

\end{document}